\documentclass[reqno]{amsart}
\usepackage{amsmath}
\usepackage{amssymb}
\usepackage{amsthm}
\usepackage{youngtab}
\usepackage{cases}
\usepackage{stmaryrd}
\usepackage{mathrsfs}
\usepackage{hyperref}
\usepackage{mathtools}
\mathtoolsset{showonlyrefs,showmanualtags}
\numberwithin{equation}{section}
\title{Schur partition theorems via perfect crystal}

\author{Shunsuke Tsuchioka}
\address{\mbox{Department of Mathematical and Computing Sciences, Tokyo Institute of Technology, Tokyo, Japan}}
\email{tshun@kurims.kyoto-u.ac.jp}

\author{Masaki Watanabe}
\address{Graduate School of Mathematical Sciences, University of Tokyo, Tokyo, Japan (current affiliation : Preferred Networks, Inc. Chiyoda-ku, Tokyo, Japan)}
\email{masakiw04@gmail.com}

\date{Feb 11, 2024}

\newtheorem{Thm}{Theorem}[section]
\newtheorem{Def}[Thm]{Definition}

\newtheorem{Prop}[Thm]{Proposition}
\newtheorem{Lem}[Thm]{Lemma}

\newtheorem{Rem}[Thm]{Remark}
\newtheorem{Cor}[Thm]{Corollary}
\newtheorem{Ex}[Thm]{Example}

\usepackage{graphicx}
\usepackage{tikz}

\def\node#1#2{\overset{#1}{\underset{#2}{\circ}}}

\def\ver#1#2{\overset{{\llap{$\scriptstyle#1$}\displaystyle\circ{\rlap{$\scriptstyle#2$}}}}{\scriptstyle\vert}}

\usetikzlibrary{arrows}
\tikzstyle{every picture}+=[remember picture]
\tikzstyle{na} = [baseline=-.5ex]
\tikzstyle{mine}= [arrows={angle 90}-{angle 90},thick]

\def\Llleftarrow{%
\lower2pt\hbox{\begingroup
\tikz
\draw[shorten >=0pt,shorten <=0pt] (0,3pt) -- ++(-1em,0) (0,1pt) -- ++(-1em-1pt,0) (0,-1pt) -- ++(-1em-1pt,0) (0,-3pt) -- ++(-1em,0) (-1em+1pt,5pt) to[out=-105,in=45] (-1em-2pt,0) to[out=-45,in=105] (-1em+1pt,-5pt);
\endgroup}
}

\def\Rrrightarrow{\vcenter{\hbox{\rotatebox{180}{\Llleftarrow}}}}

\newcommand{\hooklongrightarrow}{\lhook\joinrel\longrightarrow}

\DeclareMathOperator*{\restprod}%
{\mathchoice{\ooalign{\ensuremath{\displaystyle\prod}\crcr\ensuremath{\displaystyle\coprod}}}%
  {\ooalign{\ensuremath{\textstyle\prod}\crcr\ensuremath{\textstyle\coprod}}}%
  {\ooalign{\ensuremath{\scriptstyle\prod}\crcr\ensuremath{\scriptstyle\coprod}}}%
  {\ooalign{\ensuremath{\scriptscriptstyle\prod}\crcr\ensuremath{\scriptscriptstyle\coprod}}}%
}

\makeatletter
\renewcommand{\boxed}[2][\fboxsep]{{%
  \setlength{\fboxsep}{#1}\fbox{\m@th$\displaystyle#2$}}}
\makeatother

\DeclareMathOperator{\CL}{cl}
\newcommand{\SPE}{$\circ$}
\newcommand{\SEIGEN}{{\times'}}

\newcommand{\PRODPERFCRY}{\otimes'_{j\geq 1}S_p^{[(j-1)p+1,jp-1]}}

\newcommand{\CHECK}{\vee}
\newcommand{\RANGLE}{\rangle_{\CL}}
\newcommand{\AIZ}{i_0}
\newcommand{\HEC}{\mathcal{H}}
\newcommand{\FS}{Z}
\newcommand{\FSS}{W}
\newcommand{\CIR}{\circ}
\newcommand{\FI}{\mathbb{F}}
\newcommand{\FIF}[1]{\mathbb{F}_{#1}}

\newcommand{\GOOD}{\beta}
\newcommand{\LANGLEL}{\langle\!\langle}
\newcommand{\RANGLER}{\rangle\!\rangle}
\newcommand{\GOODZERO}[1]{\beta_{#1}}

\newcommand{\NONAME}[2]{\Theta_{#1}^{#2}}

\newcommand{\ORD}{\preceq}
\newcommand{\FORB}{\mathsf{Forb}}
\newcommand{\PARITY}{\mathsf{parity}}
\newcommand{\FORBB}[1]{\mathsf{Forb}^{\mathsf{str}}_{#1}}
\newcommand{\LONGHOOKRIGHTARROW}{\hooklongrightarrow}

\newcommand{\PAR}{\mathsf{Par}}
\newcommand{\STRICT}{\mathsf{Str}}
\newcommand{\SCHUR}{\mathsf{Schur}}

\newcommand{\PCL}{P_{\textrm{cl}}}
\newcommand{\PCLL}{P_{\normalfont{\textrm{cl}}}}
\newcommand{\PCLC}{P^{\CHECK}_{\textrm{cl}}}
\newcommand{\MC}{\mathcal{C}}
\newcommand{\MD}{\mathcal{D}}

\newcommand{\PFN}{\mathsf{Pfn}}
\newcommand{\KE}{\tilde{e}}
\newcommand{\KF}{\tilde{f}}

\newcommand{\AFP}[1]{\overline{\mathbb{F}}_{#1}}
\newcommand{\DSYM}[1]{\widehat{\mathfrak{S}}_{#1}}
\newcommand{\SYM}[1]{\mathfrak{S}_{#1}}

\newcommand{\KRC}[2]{B^{#1,#2}}
\newcommand{\KRM}[2]{W^{#1,#2}}
\newcommand{\KEP}[1]{\tilde{e}'_{#1}}
\newcommand{\KFP}[1]{\tilde{f}'_{#1}}

\newcommand{\MM}[1]{\nu_{#1}}
\newcommand{\MMM}[1]{\nu'_{#1}}

\newcommand{\SC}[1]{D^{\mathsf{odd}}_{#1}}
\newcommand{\SCC}[1]{D^{\mathsf{str}}_{#1}}
\newcommand{\AD}[1]{(A^{(2)}_{#1})^{\dagger}}
\newcommand{\NODE}[1]{\boxed[.4\fboxsep]{\mathstrut{#1}}}

\newcommand{\AR}[3]{{#1}\stackrel{#2}{\to}{#3}}
\newcommand{\LONGTO}{\longrightarrow}

\newcommand{\RES}[3]{{#1}|_{[{#2},{#3}]}}
\newcommand{\CONCAT}[2]{{{#1}\sqcup{#2}}}
\newcommand{\STHREE}[2]{\textrm{(S3)}_{#1,#2}}
\newcommand{\SSTHREE}[2]{\textup{(S3)}${}_{#1,#2}$}
\newcommand{\SSSTHREE}[2]{\normalfont{(S3)}${}_{#1,#2}$}

\newcommand{\SSONE}[2]{\textup{(S1)}\phantom{${}_{#1,#2}$}}
\newcommand{\SSTWO}[2]{\textup{(S2)}\phantom{${}_{#1,#2}$}}
\newcommand{\CTWO}[1]{\textrm{(C2)}_{#1}}
\newcommand{\CONE}[1]{\textrm{(C1)}_{\phantom{#1}}}
\newcommand{\CTHREE}[1]{\textrm{(C3)}_{#1}}
\newcommand{\CFOUR}[1]{\textrm{(C4)}_{#1}}

\newcommand{\EXS}[1]{Y_{#1}}
\newcommand{\SFOUR}[2]{\textrm{(S4)}_{#1,#2}}
\newcommand{\SSFOUR}[2]{\textup{(S4)}${}_{#1,#2}$}
\newcommand{\SSSFOUR}[2]{\normalfont{(S4)}${}_{#1,#2}$}
\newcommand{\SFOURC}{\textrm{(S4)}}

\newcommand{\SFIVE}[2]{\textrm{(S5)}_{#1,#2}}
\newcommand{\SSFIVE}[2]{\textup{(S5)}${}_{#1,#2}$}
\newcommand{\SSSFIVE}[2]{\normalfont{(S5)}${}_{#1,#2}$}
\newcommand{\SFIVEC}{\textrm{(S5)}}

\newcommand{\KEO}[1]{\tilde{e}_{1,#1}}
\newcommand{\KFO}[1]{\tilde{f}_{1,#1}}
\newcommand{\KET}[1]{\tilde{e}_{2,#1}}
\newcommand{\KFT}[1]{\tilde{f}_{2,#1}}
\newcommand{\KNB}[2]{\mathbb{B}_{#1}^{#2}}

\newcommand{\SET}{\mathsf{Set}}
\newcommand{\SETS}{\SET_{\ast}}
\newcommand{\PT}{\stackrel{\mathsf{PT}}{\sim}}
\newcommand{\ISOM}{\stackrel{\sim}{\longrightarrow}}
\newcommand{\LONGMAPSTO}{\longmapsto}
\newcommand{\ZERO}{\boldsymbol{0}}
\newcommand{\Z}{\mathbb{Z}}
\newcommand{\Q}{\mathbb{Q}}
\newcommand{\DIAMOND}{\diamond}
\newcommand{\DIAMONDD}{\bullet}
\newcommand{\INVMAP}[1]{\Psi_{#1}}
\newcommand{\INVMAPSTR}[1]{\Psi^{\mathsf{str}}_{#1}}
\newcommand{\AI}[1]{I_{#1}}
\newcommand{\AII}[1]{I^{\circ}_{#1}}

\newcommand{\CONC}{C}
\newcommand{\CONCC}{C'}

\DeclareMathOperator{\SHIFT}{\mathsf{shift}}

\DeclareMathOperator{\HOM}{Hom}
\DeclareMathOperator{\RANK}{rank}
\DeclareMathOperator{\GCD}{gcd}
\DeclareMathOperator{\NOT}{\mathsf{Not}}

\DeclareMathOperator{\IRR}{Irr}

\DeclareMathOperator{\WT}{\mathsf{wt}}
\DeclareMathOperator{\RP}{\mathsf{RP}}
\DeclareMathOperator{\REG}{\mathsf{Reg}}
\DeclareMathOperator{\CREG}{\mathsf{CReg}}
\DeclareMathOperator{\CG}{\mathsf{CG}}
\DeclareMathOperator{\AUT}{Aut}
\DeclareMathOperator{\BETAA}{\Omega}
\DeclareMathOperator{\BETAAA}{\Upsilon}
\DeclareMathOperator{\ID}{\mathsf{id}}
\DeclareMathOperator{\RAD}{\mathsf{rad}}
\DeclareMathOperator{\CHAR}{\mathsf{char}}

\begin{document}
\maketitle

\begin{abstract}
  Motivated by spin modular representations 
  of the symmetric groups, 
  we propose two generalizations of the Schur regular partitions for an odd integer $p\geq 3$. One forms a subset of the set of $p$-strict partitions, and the other forms that of strict partitions. We prove that each set has a basic $A^{(2)}_{p-1}$-crystal structure. For $p=3$, it reproves Schur's 1926 partition theorem, a mod 6 analog of Rogers-Ramanujan partition theorem (RRPT).  For $p=5$, it gives a computer-free proof of a conjecture by Andrews during his 3-parameter generalization of RRPT, which was first proved by Andrews-Bessenrodt-Olsson. 
\end{abstract}

\section{Introduction}\label{intro}
The set of partitions is denoted by $\PAR$. For a partition $\lambda=(\lambda_1,\lambda_2,\dots,\lambda_{\ell})\in\PAR$, we define
$|\lambda|=\lambda_1+\cdots+\lambda_{\ell}$, $m_i(\lambda)=|\{1\leq j\leq\ell\mid\lambda_j=i\}|$ for $i\in\Z$, and $\ell(\lambda)=\ell$. 
The symbol $\varnothing$ is reserved for the empty partition. 
For $a\geq 1$, we define
\begin{enumerate}
\item $\STRICT_a$, the set of $a$-strict partitions (i.e., $m_i(\lambda)>1$ implies $i\in a\Z_{\geq 1}$),
\item $\REG_a$, the set of $a$-regular partitions (i.e., $m_i(\lambda)<a$ for $i\geq 1$),
\item $\CREG_a$, the set of $a$-class regular partitions (i.e., $m_{ai}(\lambda)=0$ for $i\geq 1$),
\end{enumerate}
and define $\STRICT=\REG_2$ to be the set of strict partitions (a.k.a., distinct partitions).


\begin{Def}\label{PSI}
We say that a partition $\lambda$ satisfies the $p$-Schur inequality if
we have
$\lambda_i-\lambda_{i+h}\geq p$ for $1\leq i\leq\ell(\lambda)-h$, and the inequality is strict if $\lambda_i\in p\mathbb{Z}$.
\end{Def}

For $a\geq 1$, an $a$-step shift 
of $\lambda$
is defined as $\SHIFT^k_a(\lambda)=(\lambda_1+ka,\dots,\lambda_\ell+ka)$, where $k\in\mathbb{Z}$ and $\lambda_{\ell}+ka>0$. 
We say that $\lambda$ contains a subpattern $\mu=(\mu_1,\dots,\mu_{\ell'})\in\PAR$ if there exists 
$0\leq i\leq \ell-\ell'$ such that $\lambda_{i+j}=\mu_{j}$ for $1\leq j\leq \ell'$.

\begin{Def}\label{DEFSP}
For an odd integer $p=2h+1\geq 3$, we denote by $S_p$ the set of $p$-strict partitions
$\lambda\in\STRICT_p$ which satisfy the $p$-Schur inequality
and do not contain any subpatterns $\SHIFT^k_p(\mu)$, where $\mu\in \FORB_p$ and $k\geq 0$.
\end{Def}

Here, $\FORB_p$ is a finite set
consisting of the following partitions.
\begin{center}
\begin{tabular}{ccc}
$(p+1,p-1)$, & $(h+1,h)$, & $(2p+1,\ast^h,p-1)$, \\
$(p+2,\ast,p-2)$, & $(h+2,\ast,h-1)$, & $(2p+2,\ast^{h+1},p-2)$, \\
$\vdots$ & $\vdots$ & $\vdots$ \\
  $(p+h-1,\ast^{h-2},p-(h-1))$, &
  $(2h-1,\ast^{h-2},2)$, &
  $(2p+h-1,\ast^{2h-2},p-(h-1))$.
\end{tabular}
\end{center}
A wild card $\ast$ is a placeholder where we can fill any integer without violating $p$-strictness, and $\ast^b$ stands for the $b$-repetition of $\ast$.
Note that $\FORB_3=\emptyset$.

\subsection{The main results}
We have two goals.
One is to prove Theorem \ref{maintheorem}, and the other is
to define $\SCHUR_p\subseteq\STRICT$ which has similar properties (see Theorem \ref{maintheorem3}).

\begin{Def}[{\cite[Definition 3]{An3}}]\label{PTDEF}
  We say that subsets $\MC,\MD\subseteq\PAR$ are partition theoretically equivalent (written as $\MC\PT\MD$) if $|\MC(n)|=|\MD(n)|$ for $n\geq 0$, 
  where $\MC(n)=\{\lambda\in\MC\mid|\lambda|=n\}$ is the set of partitions of $n$ in $\MC$, etc.
\end{Def}


\begin{Thm}\label{maintheorem}
For an odd integer $p\geq 3$, we have $S_p\PT\SC{p}$ (resp. $S_p\PT\SCC{p}$).
\end{Thm}

Here, $\SC{p}$ (resp. $\SCC{p}$) is the set of odd (resp. strict) $p$-class regular partitions,
i.e., $\SC{p}=\CREG_2\cap\CREG_p$ (resp. $\SCC{p}=\STRICT\cap\CREG_p$).
Note that it is not difficult to see $\SC{p}\PT\SCC{p}$ for an odd integer $p\geq 3$.

The partition theorems in the main results are related to the Rogers-Ramanujan identities
and spin modular representations of the symmetric groups, as we will see in \S\ref{anRRPT} and in \S\ref{expschur}, \S\ref{brakcr}, respectively.
The statements for $p=3,5$ are known 
and located in a distinguished place in the theory of partitions~\cite[\S6]{An2}~\cite{An4}. 

\subsection{Andrews' 3-parameter generalization of RRPT}\label{anRRPT}
Throughout the paper, we abbreviate Rogers-Ramanujan (resp. partition theorem) with RR (resp. PT).

In the first letter to Hardy, Ramanujan wrote an astonishing formula
\begin{align*}
\cfrac{1}{1+\cfrac{e^{-2\pi}}{1+\cfrac{e^{-4\pi}}{\vdots}}}
=\left(\sqrt{\frac{5+\sqrt{5}}{2}}-\frac{\sqrt{5}+1}{2}\right)e^{2\pi/5},
\end{align*}
which is now known as
the RR continuous fraction.
For the history, we refer ~\cite[\S7]{An1},~\cite[\S1]{An2}, ~\cite{Har} and ~\cite{Sil}.
A proof can be reduced to the RR identities~\cite{Rog}.
As observed by Schur and MacMahon, they are equivalent to RRPT.


\begin{Thm}[RRPT]\label{RRPT}
We have $R\PT T^{(5)}_{1,4}$ and $R'\PT T^{(5)}_{2,3}$, where
\begin{align*}
R &= \{\lambda\in \PAR\mid
 \textrm{$\lambda_i-\lambda_{i+1}\geq 2$ for $1\leq i<\ell(\lambda)$}\},\\
R' &= \{\lambda\in R\mid m_1(\lambda)=0 \textrm{ (i.e., $\lambda_{\ell(\lambda)}\geq 2$ if $\lambda\ne\varnothing$)}\}, \\
T^{(N)}_{a,\dots,b} &= \{\lambda\in\PAR\mid \textrm{$m_i(\lambda)>0$ implies $i\equiv a,\dots,b\!\!\pmod{N}$}\}.
\end{align*}
\end{Thm}


Based on a classical method~\cite[p.1037]{An3}, which utilizes
the 2-parameter generating functions, 
Andrews established a 3-parameter RRPT.
The parameter $\ell=0$, $k=2$, $a=2$ (resp. $a=1$) duplicates $R\PT T^{(5)}_{1,4}$ (resp. $R'\PT T^{(5)}_{2,3}$).

\begin{Thm}[{\cite{An4}}]\label{AndrewGen}
For $\ell,k,a\geq 0$ with $0\leq \ell/2<a\leq k\geq \ell$, we have $A_{\ell,k,a}\PT B_{\ell,k,a}$, 
where $B_{\ell,k,a} = \{\lambda\in\STRICT_{\ell+1}\mid\normalfont{\text{(B1), (B2)}}\}$, and 
\begin{align*}
A_{\ell,k,a} = \begin{cases}
\{\lambda\in\STRICT_{\ell+1}\mid\normalfont{\textrm{(A1)}}\} & \textrm{if $\ell$ is even}, \\
\{\lambda\in\STRICT_{(\ell+1)/2}\mid\normalfont{\text{(A1), (A2)}}\} & \textrm{if $\ell$ is odd}.
\end{cases}
\end{align*}
\begin{enumerate}
\item[(A1)] $m_i(\lambda)=0$ for $i\equiv 0,\pm(2a-\ell)(\ell+1)/2\pmod{(2k-\ell+1)(\ell+1)}$.
\item[(A2)] $m_i(\lambda)=0$ for $i\equiv \ell+1\pmod{2(\ell+1)}$.
\item[(B1)] $\lambda_i-\lambda_{i+k-1}\geq \ell+1$ for $1\leq i\leq \ell(\lambda)-(k-1)$, and the inequality is strict 
if $\lambda_i\in(\ell+1)\Z$.
\item[(B2)] $\sum_{i=1}^{\ell+1}m_i(\lambda)\leq a-1$ and $\sum_{i=j}^{\ell-j+1}m_i(\lambda)\leq a-j$ for $1\leq j\leq (\ell+1)/2$.
\end{enumerate}
\end{Thm}

Andrews observed that even if the condition $0\leq \ell/2<a\leq k\geq \ell$ is violated,
we still have a chance to have a PT $A_{\ell,k,a}\PT B^{\CIR}_{\ell,k,a}$ by defining a suitable 
$B^{\CIR}_{\ell,k,a}\subseteq B_{\ell,k,a}$. As an example, see ~\cite[p.84]{An4} for $\ell=3$, $k=a=2$.

\begin{Thm}[{\cite[Conjecture 2]{An4}}$=${\cite[Theorem 3.1]{ABO}}]\label{AndrewCon}
We have $(A_{4,3,3}=)\SCC{5}\PT B^{\CIR}_{4,3,3}$, where 
$B^{\CIR}_{4,3,3}:=\{\lambda\in\STRICT_5\mid\textrm{\normalfont{(C1), $\CTWO{j}$, $\CTHREE{j}$, $\CFOUR{j}$} for $j\geq 0$}\}$. 
\begin{enumerate}
\item[$\CONE{j}$]\label{ancon1} $\lambda_i-\lambda_{i+2}\geq 5$ for $1\leq i\leq \ell(\lambda)-2$, and the inequality is strict 
  if $\lambda_i\in 5\mathbb{Z}$.
\item[$\CTWO{j}$]\label{ancon2} $m_{5j+3}(\lambda)+m_{5j+2}(\lambda)\leq1$. 
\item[$\CTHREE{j}$]\label{ancon3} $m_{5j+6}(\lambda)+m_{5j+4}(\lambda)\leq1$. 
\item[$\CFOUR{j}$]\label{ancon4} $m_{5j+11}(\lambda)+m_{5j+10}(\lambda)+m_{5j+5}(\lambda)+m_{5j+4}(\lambda)\leq3$. 
\end{enumerate}
\end{Thm}

Note that $B^{\CIR}_{4,3,3}=S_5$ 
and $B_{4,3,3}=\{\lambda\in\STRICT_5\mid\textrm{(C1), $\CTWO{0}$}\}$.
Theorem \ref{AndrewCon} was proposed as a conjecture in the last section of ~\cite{An4} with
a concluding phrase, ``the methods of proof should have interesting ramifications in the theory of partition identities''.
Nearly 20 years later, Andrews-Bessenrodt-Olsson established Theorem \ref{AndrewCon}
by using a truncated approximation of the 2-parameter generating function.
The method is flexible (e.g., see an application to G\"ollnitz PT~\cite[\S10.6]{An2}), yet
it requires heavy computer usage (see formulas in ~\cite[\S3]{ABO}).

Theorem \ref{AndrewCon} is the same as Theorem \ref{maintheorem} for $p=5$. Thus, Theorem \ref{maintheorem} gives not only
a 
computer-free proof of Theorem \ref{AndrewCon}, but also
a definition of $B^{\CIR}_{p-1,h+1,h+1}:=S_p\subseteq B_{p-1,h+1,h+1}$ (see Proposition \ref{andsub}) for $p=2h+1\geq 3$ such that we have $(\SCC{p}=)A_{p-1,h+1,h+1}\PT B^{\CIR}_{p-1,h+1,h+1}(=S_p)$.


\subsection{Schur PT and $p$-modular representations of the Schur cover $\DSYM{n}$}\label{expschur}
We mention an incarnation of Theorem \ref{maintheorem} for $p=3,5$ other than RRPT.
In 1926, Schur found a PT, which is duplicated as Theorem \ref{AndrewGen} for $\ell=k=a=2$. 

\begin{Thm}[{\cite[Satz V]{Sch}}]\label{SchurPT}
  We have $S_{3}\PT\SCC{3}$.
\end{Thm}

For a proof by bijection, see ~\cite{Be1,Bre}. 
Because of $\FORB_3=\emptyset$, the statement $S_3\PT T^{(6)}_{1,5}(=\SC{3}\PT\SCC{3})$
can be regarded as a mod 6 analog of RRPT $R\PT T^{(5)}_{1,4}$.


Let $\SYM{n}=\AUT(\{1,\dots,n\})$ be the symmetric group of $n$ letters. 
For $\lambda\in\PAR$, one can construct a $\Z$-free of finite rank,  $\Z\SYM{n}$-module $S^{\lambda}$ with an $\SYM{n}$-invariant bilinear form
$\langle,\rangle:S^{\lambda}\times S^{\lambda}\to\Z$ (Specht module~\cite{Jam}). The construction gives bijections
\begin{align}
\begin{split}
\PAR(n) &\ISOM \IRR(\mathbb{Q}\SYM{n}),\quad\lambda\LONGMAPSTO \Q\otimes S^{\lambda}, \\
\REG_p(n) &\ISOM \IRR(\FIF{p}\SYM{n}),\quad\lambda\LONGMAPSTO D^{\lambda}_{\FIF{p}}:=(\FIF{p}\otimes S^{\lambda})/\RAD \langle,\rangle.
\end{split}
\label{parasym}
\end{align}
Here, $\IRR(\FI\SYM{n})$ for a field $\FI$ stands for the set of isomorphism classes of irreducible $\FI\SYM{n}$-modules.
The Glaisher bijection $\REG_p\PT\CREG_p$ restricted to a prime $p\geq 2$ can be understood as 
an incarnation of a theorem of Brauer (see ~\cite[(15.11)]{Isa}).

The parameterization \eqref{parasym} has nice features. For example, for a prime $p\geq 2$, 
the $p$-decomposition matrix of $\SYM{n}$ is regularized as the form
\begin{align*}
\begin{pmatrix}
1      & {}       & O \\
{}     & {\ddots} & {} \\
{\ast} & {}       & 1  \\ \hline
{} & {\ast}   & {}
\end{pmatrix},
\end{align*}
where the row is indexed by $\REG_p(n)$ in lexicographically decreasing order followed by $\PAR(n)\setminus\REG_p(n)$, and the column is again indexed by $\REG_p(n)$ in lexicographically decreasing order~\cite{FMP}. 

Let $\DSYM{n}$ be the Schur cover of $\SYM{n}$ (see ~\cite[\S1]{Ste}) 
with the short exact sequence
\begin{align*}
1
\LONGTO
\langle z\mid z^2=1\rangle
\LONGTO
\DSYM{n}
\LONGTO
\SYM{n}
\LONGTO
1.
\end{align*}

In the rest, $\IRR(\FI\DSYM{n})$ for a field $\FI$ with $\CHAR\FI\ne2$ stands for the set of isomorphism classes (modulo association, see ~\cite[\S4]{Ste}, ~\cite[\S2.5]{KKT}) 
of irreducible $\FI\DSYM{n}$-modules, where $z$  acts as $-1$.
Since Schur established the parameterization $\IRR(\mathbb{C}\DSYM{n})\cong \STRICT(n)$~\cite{Sch2}, it is natural to expect $\SCHUR_p\subseteq\STRICT$ with the following properties.
\begin{enumerate}
\item We have a bijection $\IRR(\AFP{p}\DSYM{n})\cong \SCHUR_p(n)$, which implies $\SCHUR_p\PT\SC{p}$ by a version of the Brauer theorem (see ~\cite[\S3]{Be2}).
\item The parameterization is nice, e.g., an analogous 
regularization result for the $p$-modular decomposition matrix of $\DSYM{n}$ holds modulo power of 2.
\end{enumerate}

Bessenrodt-Morris-Olsson showed that 
$\SCHUR_3=S_3$ satisfies the properties above
for $p=3$~\cite[Theorem 4.5]{BMO}. 
Furthermore, for $p=5$, they found a class $\SCHUR_5$.

\begin{Thm}[{\cite[\S5, Conjecture]{BMO}}$=${\cite[Theorem 3.1]{ABO}}]\label{BMOGen}
We have $\SCHUR_{5}\PT\SC{5}$, where 
$\SCHUR_5$ is a set of strict partitions $\lambda\in\STRICT$ such that
\begin{enumerate}
\item[(D1)] $\lambda_i-\lambda_{i+2}\geq 5$ for $1\leq i\leq \ell(\lambda)-2$, and the inequality is strict if $\lambda_i\in 5\mathbb{Z}$ or $\lambda_i+\lambda_{i+1}\in 5\mathbb{Z}$,
\item[(D2)] $\lambda$ does not contain $\SHIFT^k_5(\mu)$ for $k\geq 0$ and
$\mu\in\{(3,2), (11,9,5), (10,6,4), (11,10,5,4)\}$.
\end{enumerate}
\end{Thm}

Theorem \ref{BMOGen} appeared as a concluding conjecture of ~\cite{BMO} with
a notice $(B^{\CIR}_{4,3,3}=)S_5\PT\SCHUR_5$ via a size preserving bijection
$\BETAA_5:S_5\ISOM\SCHUR_5$. Here, $\BETAA_5(\lambda)$ for $\lambda\in S_5$ is defined 
by replacing every occurrence of $(5k,5k)$ 
with $(5k+1,5k-1)$ for $k\geq 1$. Thus, Theorem \ref{AndrewCon} and Theorem \ref{BMOGen} are
equivalent. 
Furthermore, a weaker regularization result analogous to ~\cite{BMO} holds (see ~\cite[Theorem 3.10]{Be2}).
Subsequently,
a definition of $\SCHUR_7$ was searched, resulting in partial results (see ~\cite{ABO,Be2,LT}).


\subsection{Modular branching rules and Kashiwara crystal theory}\label{brakcr}
Lascoux-Leclerc-Thibon-Ariki theory (see ~\cite{Kl2} and references therein)
started from Leclerc's observation that 
Kleshchev's modular branching rule is an $A^{(1)}_{p-1}$-crystal isomorphism
\begin{align}
{\textstyle{\bigsqcup_{n\geq 0}}}\IRR(\FIF{p}\SYM{n})\cong \REG_p(\cong B(\Lambda_0)),\quad
D^{\lambda}_{\FIF{p}}\mapsfrom\lambda.
\label{MMrea}
\end{align}
Here, $p\geq 2$ is a prime, and $\REG_p$ has an $A^{(1)}_{p-1}$-crystal structure through Misra-Miwa realization~\cite{MM}.
For Kashiwara crystal theory, see a brief review in \S\ref{kcsec}.

Later, relying on a characterization of crystals due to Kashiwara-Saito~\cite[Proposition 3.2.3]{KS},
Grojnowski devised a method of establishing a connection between Kashiwara crystals and a tower of ``Hecke algebras'' $\{\HEC_n\}_{n\geq0}$~\cite{Gr1,GV}.
Based on that, Brundan-Kleshchev established a modular branching rule for $\DSYM{n}$.
\begin{Thm}[{\cite[Theorem 8.11]{BK}}]\label{BKthm}
The set $\bigsqcup_{n\geq 0}\IRR(\AFP{p}\DSYM{n})$ has an $A^{(2)}_{p-1}$-crystal structure (see Figure \ref{twisted}), which is isomorphic to $B(\Lambda_0)$
for an odd prime $p\geq 3$.
\end{Thm}


The method does not require a detail of 
$\IRR(\HEC_n)$ in advance,
and starts from a branching rule of $\bigsqcup_{n\geq0}\IRR(\HEC_n)$ 
in a spirit similar to Okounkov-Vershik~\cite{OV}.
To get concrete as \eqref{MMrea},
Kyoto path realization via perfect crystal~\cite{KKMMNN2,KKMMNN} (see \S\ref{PERF}) is used. 
An application gives an $A^{(2)}_{p-1}$-crystal isomorphism
\begin{align*}
{\textstyle{\bigsqcup_{n\geq 0}}}\IRR(\AFP{p}\DSYM{n})\cong \RP_p(\cong B(\Lambda_0)),
\end{align*}
where $\RP_p$ is a set of $p$-strict $p$-restricted partitions in the sense of ~\cite[\S9-a]{BK}. Namely,
$\RP_p$ is a set of $p$-strict partitions $\lambda$ such that
$\lambda_i-\lambda_{i+1}\leq p$ for $1\leq i\leq\ell(\lambda)$, and
the inequality is strict if $\lambda_i\in p\Z$ (see also ~\cite{Kan}),
where we promise $\lambda_i=0$ for $i>\ell(\lambda)$.

The parameterization by $\RP_p$, proposed first in ~\cite[\S7]{LT}, was obtained by a different method~\cite{BK2}. 
Moreover, a certain regularization result holds~\cite{BK3}.
While the explanation so far 
may be a reason for the abandonment of searching for
the aforementioned $\SCHUR_p\subseteq\STRICT$,
it gives us a clue to prove Theorem \ref{maintheorem} as follows.

\begin{Thm}\label{maintheorem2}
For an odd integer $p=2h+1\geq 3$, the set $S_p$ has an $A^{(2)}_{p-1}$-crystal structure, which satisfies the two conditions below.
\begin{enumerate}
\item\label{maintheorem2a} $S_p$ is isomorphic to the basic crystal $B(\Lambda_0)$.
\item\label{maintheorem2b} $\KF_i\lambda$ is obtained by adding a residue $i$-box to $\lambda$ 
  if $\KF_i\lambda$ is defined for $\lambda\in S_p$ and $i\in \AI{p}=\{0,1,\dots,h\}$.
\end{enumerate}
\end{Thm}


Though the proof of Theorem \ref{maintheorem2} is an application of Kyoto path realization, 
we would like to 
emphasize that the method is different from what a specialist might imagine. 
A natural way to realize an affine crystal as a subset of $\PAR$
uses a perfect crystal that does not have a branching vertex in the sense of ~\cite[\S1]{Tsu}. 
The aforementioned $A^{(2)}_{p-1}$-crystal isomorphism $\RP_p\cong B(\Lambda_0)$ is an example. 
However, 
Kirillov-Reshetikhin perfect crystal $\KRC{h}{2}$ for $\AD{p-1}$ (see \S\ref{PERF}), which we use for the proof, 
has many branching vertices. 
We remark that $\RP_p\PT S_p$ can be understood as a kind of transpose of Young diagrams (for $p=3$, see ~\cite[\S3]{Be1}).

\subsection{A definition of $\SCHUR_p$ and future directions}
\label{finalsubsection}
Note that $S_p$ is not a subset of $\STRICT$ for $p\ne3$.
We define a size and length preserving injection $\GOOD_p:S_p^{[p-h,p+h]}\hookrightarrow \STRICT^{[p-h,p+h]}$ for $p=2h+1\geq 3$ (for notation, see \S\ref{noco}), and 
define $\SCHUR_p\subseteq\STRICT$ as ``the image of $S_p$'' (see \S\ref{defsch1}, \S\ref{defsch2}).
This generalizes $\SCHUR_3=S_3$ and $\SCHUR_5$ (see \S\ref{exschur}). The results for $S_p$ have $\SCHUR_p$ versions as follows.
\begin{Thm}\label{maintheorem3}
  For an odd integer $p\geq 3$, $\SCHUR_p$ has the following properties.
  \begin{enumerate}
  \item\label{maintheorem3a} Any partition in $\SCHUR_p$ satisfies the $p$-Schur inequality.
  \item\label{maintheorem3b} $\SCHUR_p$ is characterized in terms of pattern avoidance, i.e., there is a finite subset $\FORBB{p}\subseteq\STRICT$ such that
    $\lambda\in\STRICT$ belongs to $\SCHUR_p$ if and only if
    $\lambda$ does not contain
any subpatterns $\SHIFT^k_p(\mu)$, where $\mu\in \FORBB{p}$ and $k\geq 0$.
    \item\label{maintheorem3c} Theorem \ref{maintheorem2} holds when replacing $S_p$ with $\SCHUR_p$ in the statement.
  \end{enumerate}
\end{Thm}

In particular, we have $\SCHUR_p\PT S_p$ (see also Remark \ref{remscpPT}), and $\IRR(\AFP{p}\DSYM{n})$ has a parameterization by $\SCHUR_p(n)$ compatible with the modular branching rule.
It is interesting if a regularization theorem holds for this label
(see also ~\cite[\S8.4]{Cra}).

Finally, it is natural to ask if we can interpret $\SCHUR_p$ by Lepowsky-Wilson machinery~\cite{LW}, which gives a Lie theoretic interpretation of RRPT by the level 3 standard modules of the affine Lie algebra of type $A^{(1)}_1$. A conjectural statement would be that the set of $Z$-monomials
\begin{align*}
\{Z_{-\lambda_1}\cdots Z_{-\lambda_{\ell}}v_0\mid (\lambda_1,\dots,\lambda_{\ell})\in\SCHUR_p\}
\end{align*}
is a basis of the vacuum space $\Omega(V(\Lambda_0+\Lambda_{p-1}))$ with respect to the principal Heisenberg subalgebra for the affine Lie algebra of type $D^{(2)}_{p}$ (see Figure \ref{twisted}, regarding $D_3=A_3$).
Here, $v_0$ is a highest weight vector of the standard module $V(\Lambda_0+\Lambda_{p-1})$, and we put $Z_{2n}=Z_{2n}(\alpha_p)$ and $Z_{2n+1}=Z_{2n+1}(\alpha_{1})$ for $n\in\mathbb{Z}$ in the notation of ~\cite[\S9.3(ii)(3)]{Fig} and ~\cite[\S2]{Tsu2} (See also ~\cite{Ito} on a study for $A^{(2)}_{\textrm{odd}}$ level 2).
It is interesting if  $\SCHUR_p$ is related to $D^{(2)}_p$ through vertex operators, as well as $A^{(2)}_{p-1}$ through Kashiwara crystals.


\subsection{Notations and Conventions}\label{noco}
Let $a\leq b$ integers. For $\MC\subseteq\PAR$, we define $\MC^{[a,b]}=\{\lambda\in \MC\mid \textrm{$m_i(\lambda)>0$ implies $a\leq i\leq b$}\}$. For $\lambda\in\PAR$, we denote by $\RES{\lambda}{a}{b}$ a partition $\mu$ defined by $m_i(\mu)=m_i(\lambda)$ if $a\leq i\leq b$, and $m_i(\mu)=0$ if otherwise.

For $\lambda\in \PAR^{[a,b]}$ and $\mu\in\PAR^{[c,d]}$ such that $[a,b]\cap[c,d]=\emptyset$, we denote by
$\lambda\sqcup\mu$ a partition $\nu$ such that
$m_i(\nu)=m_i(\lambda)+m_i(\mu)$
for $i\geq 1$.

For $\lambda\in\STRICT$ and $i\geq 1$, $i\not\in\lambda$ (resp. $i\in\lambda$) stands for $m_i(\lambda)=0$ (resp. $m_i(\lambda)=1$), and $\lambda\sqcup\{i\}$ (resp. $\lambda\setminus\{i\}$) means a strict partition $\nu$ such that $m_j(\nu)=m_j(\lambda)+\delta_{i,j}$ (resp. $m_j(\nu)=m_j(\lambda)-\delta_{i,j}$) for $j\geq 1$. 

For a sequence $(\MC_n)_{n\geq1}$ with 
$\varnothing\in\MC_n\subseteq\PAR$, a restricted direct product 
$\times'_{n\geq1}\MC_n=\cdots\SEIGEN\MC_2\SEIGEN\MC_1$ stands for 
the set $\{(\lambda_n)_{n\geq 1}\in\prod_{n\geq 1}\MC_n\mid
\textrm{$\lambda_n=\varnothing$ for sufficiently large $n$}\}$.


\hspace{0mm}

\noindent{\bf Organization of the paper.} In \S\ref{schurdef}, 
we study $S_p(\subseteq\STRICT_p)$ introduced in Definition \ref{DEFSP}.
The zigzag property in \S\ref{zzsec} is a key result, and based on that, we put a precrystal structure on $S_p$ in \S\ref{pkcsec}.
Theorem \ref{finonebox} is another key result for the proof.
By these combinatorial results and Kashiwara crystal theory reviewed in \S\ref{kcsec}, 
we prove Theorem \ref{maintheorem2} in \S\ref{pfsec}, which implies Theorem \ref{maintheorem}. Finally, in \S\ref{defsch},
by modifying the map $\INVMAP{j}$ in \S\ref{magicmapsec},
we define $\SCHUR_p(\subseteq\STRICT)$, whose
representation theoretic motivation was explained in \S\ref{expschur}, and we prove Theorem \ref{maintheorem3}.

\hspace{0mm}

\noindent{\bf Acknowledgments.}
The authors thank Masato Okado for help them understand perfect crystals better.
The first (resp. second) author was
supported by the Research Institute for Mathematical
Sciences, an International Joint Usage/Research Center located in Kyoto
University, JSPS Kakenhi 26800005, 17K14154, 20K03506, 23K03051,
Inamori Foundation, JST CREST Grant Number JPMJCR2113, Japan and Leading Initiative for Excellent Young Researchers, MEXT, Japan (resp. JSPS Kakenhi 15J05373).


\section{Combinatorics of $S_p$}\label{schurdef}
\subsection{Definition and routine checks}
We begin by an equivalent definition of $S_p$. 

\begin{Def}\label{sdef}
For an odd integer $p=2h+1\geq 3$, the subset $S_p\subseteq\PAR$ consists of $\lambda\in\PAR$ with
the following conditions, where $1\leq d\leq h$, $0\leq d'\leq h$, $k\geq 1$, $k'\geq 0$.
\begin{center}
\begin{tabular}{ll}
\SSONE{k\phantom{'}}{d\phantom{'}} & $\lambda$ is $p$-strict (i.e., $m_{i}(\lambda)>1$ implies $i\in p\mathbb{Z}_{\geq 1}$).\\
\SSTWO{k\phantom{'}}{d\phantom{'}} & $\lambda_i-\lambda_{i+h}\geq p$ for $1\leq i\leq \ell(\lambda)-h$.\\
\SSTHREE{k\phantom{'}}{d\phantom{'}} & there is no $i$ such that $\lambda_i=pk+d,\lambda_{i+d}=pk-d$.\\
\SSFOUR{k'}{d\phantom{'}} & there is no $i$ such that $\lambda_i=pk'+h+d,\lambda_{i+d}=pk'+h+1-d$.\\
\SSFIVE{k\phantom{'}}{d'} & there is no $i$ such that $\lambda_i=pk+p+d',\lambda_{i+(h+d)}=pk-d'$.
\end{tabular}
\end{center}
\end{Def}

\begin{Rem}\label{redudef}
  Definition \ref{sdef} is equivalent to Definition \ref{DEFSP}. In fact,
  the $p$-Schur inequality is the same as \normalfont{(S2)} and \SSSFIVE{k}{0} for all $k\geq 1$.
Note also that \SSSTHREE{k}{h}, \SSSFOUR{k'}{h} and \SSSFIVE{k}{h} follow from \normalfont{(S2)}.
\end{Rem}


The following five lemmas are immediate.
Lemma \ref{invshift} (resp. Lemma \ref{invweakideal2}) states that $S_p$ is closed
under $p$-step shifts (resp. reversal).
Lemma \ref{weakideal} states that a weak form of closure property holds, although $S_p$ is not a partition ideal~\cite[\S2]{An3} if $p\geq 7$.


\begin{Lem}\label{ideal3}
We have $S_p^{[kp-h,kp-1]}=\STRICT^{[kp-h,kp-1]}$ for $k\geq 1$, and $S_p^{[jp+1,jp+h]}=\STRICT^{[jp+1,jp+h]}$ for $j\geq 0$.
\end{Lem}

\begin{Lem}\label{invshift}
We have $\SHIFT^k_p(\lambda)\in S_p$ for $\lambda\in S_p$ and $k\in\mathbb{Z}$ if defined.
\end{Lem}


\begin{Lem}\label{invweakideal2}
For $a\leq b$, $k\geq 1$ and $\lambda\in S^{[a,b]}_p$ such that $pk>\lambda_1$, 
we have $\mu\in S_p$, where $\mu$ is defined by 
$\ell(\mu)=\ell(\lambda)(=:\ell)$ and $\mu_i+\lambda_{\ell+1-i}=pk$ for $1\leq i\leq \ell$.
\end{Lem}


\begin{Lem}\label{weakideal}
For any integers $a\leq b$ and $\lambda\in S_p$, we have $\RES{\lambda}{a}{b}\in S_p$.
\end{Lem}


\begin{Lem}\label{invweakideal}
Let $\lambda\in\PAR$. If $\RES{\lambda}{(j-1)p-h}{jp+h}\in S_p$ for all $j\geq 1$, then we have $\lambda\in S_p$.
\end{Lem}


$\SFOURC$ and $\SFIVEC$ can be strengthened as follows.

\begin{Prop}\label{indnum}
Let $\lambda\in S_p$. We have 
\begin{enumerate}
\item\label{indnuma} $\sum_{k=jp+h+1-d}^{jp+h+d}m_k(\lambda)\leq d$ for $1\leq d\leq h$ and $j\geq 0$.
\item\label{indnumb} $\sum_{k=jp-d}^{jp+p+d}m_k(\lambda)\leq h+d$ for $0\leq d\leq h$ and $j\geq 0$.
\end{enumerate}
\end{Prop}

\begin{proof}
We prove (\ref{indnuma}), (\ref{indnumb}) by induction on $d$.
The base case $d=1$ for (\ref{indnuma}) follows from $\SFOUR{j}{1}$. For $2\leq d\leq h$, 
assume that (\ref{indnuma}) holds for $d-1$. If $\sum_{k=jp+h+1-(d-1)}^{jp+h+(d-1)}m_k(\lambda)\leq d-2$, then (\ref{indnuma}) holds for $d$ by the $p$-strictness of $\lambda$. 
Since $\sum_{k=jp+h+1-(d-1)}^{jp+h+(d-1)}m_k(\lambda)= d-1$ and $m_{jp+h+1-d}(\lambda)=m_{jp+h+d}(\lambda)=1$ violate $\SFOUR{j}{d}$, (\ref{indnuma}) holds in this case, too.
The base case $d=0$ for (\ref{indnumb}) follows from (S2) and $\SFIVE{j}{0}$. Then, the rest of the argument is the same as in the proof of (\ref{indnuma}), 
replacing the role of $\SFOURC$ with $\SFIVEC$.
\end{proof}

As we mentioned in \S\ref{anRRPT}, Theorem \ref{maintheorem} gives a definition of $B^{\CIR}_{p-1,h+1,h+1}$.
\begin{Prop}\label{andsub}
$S_p\subseteq B_{p-1,h+1,h+1}$ (see Theorem \ref{AndrewGen}).
\end{Prop}

\begin{proof}
(B2) follows from (S2) and
$j=0$ of Proposition \ref{indnum} (\ref{indnuma}).
\end{proof}

\subsection{The map $\lambda\mapsto(\lambda^+,\lambda^-)$}\label{magicmapsec}
In this subsection, we fix $j\geq 0$ and put 
\begin{align*}
\EXS{j}=\{\lambda\in\STRICT_p^{[jp-h,jp+h]}\mid\ell(\lambda)\leq h\},
\end{align*}
as an auxiliary set for constructing a bijection $S_p^{[jp-h,jp+h]}\cong S_p^{[jp+1,jp+h]}\times S_p^{[jp-h,jp-1]}$.

\begin{Def}\label{magicmap}
For $\lambda\in \EXS{j}$, we put 
\begin{align}
\begin{split}
a_i=a_i(\lambda) &:= -i+\sum_{k=jp-i}^{jp+i}m_k(\lambda)\quad\textrm{for\quad$0\leq i\leq h$},\\
i_0=i_0(\lambda) &:= \min \{0\leq i\leq h\mid a_i=0\},
\label{henkanm2}
\end{split}
\end{align}
and define $\lambda^+\in S_p^{[jp+1,jp+h]}$ and $\lambda^-\in S_p^{[jp-h,jp-1]}$ (see Lemma \ref{ideal3}) by
\begin{align}
m_{jp\pm i}(\lambda^{\pm}) = \begin{cases}
1-m_{jp\mp i}(\lambda) & \textrm{if $1\leq i\leq i_0$}, \\
m_{jp\pm i}(\lambda) & \textrm{if $i_0< i\leq h$}.
\end{cases}
\label{henkanm}
\end{align}
\end{Def}

Note that $i_0(\lambda)$ is well-defined by $a_0(\lambda) = m_{jp}(\lambda) \geq 0$, $a_h(\lambda) = \ell(\lambda)-h \leq 0$ and
\begin{align}
a_i(\lambda)-a_{i-1}(\lambda)=m_{jp+i}(\lambda)+m_{jp-i}(\lambda)-1 \in\{-1,0,1\}.
\label{aexreason}
\end{align}

\begin{Rem}\label{consteng}
Put another way, $\lambda^\pm$ is constructed by the following rule.
\begin{enumerate}
\item For $1\leq i\leq i_0$, we have, for any $(\DIAMOND,\DIAMONDD)=(\in,\not\in), (\not\in,\in)$,
\begin{itemize}
\item[{}] $jp+i\DIAMOND\lambda^+$ and $jp-i\DIAMOND\lambda^-$, if $jp+i\DIAMONDD\lambda$ and $jp-i\DIAMONDD\lambda$,
\item[{}] $jp+i\DIAMOND\lambda^+$ and $jp-i\DIAMONDD\lambda^-$, if $jp+i\DIAMOND\lambda$ and $jp-i\DIAMONDD\lambda$.
\end{itemize}
\item For $i_0<i\leq h$, we have, for any $(\DIAMOND,\DIAMONDD)=(\in,\in), (\in,\not\in), (\not\in,\in), (\not\in,\not\in)$,
\begin{itemize}
\item[{}] $jp+i\DIAMOND\lambda^+$ and $jp-i\DIAMONDD\lambda^-$, if $jp+i\DIAMOND\lambda$ and $jp-i\DIAMONDD\lambda$.
\end{itemize}
\end{enumerate}
\end{Rem}

\begin{Prop}\label{charSp}
We have
  \begin{align*}
  S_p^{[jp-h,jp+h]}=\{\lambda\in \EXS{j}\mid
  \textrm{$(a_{d-1}(\lambda),a_d(\lambda))\ne(0,1)$
  for $1\leq d\leq h$}\}.
\end{align*}
\end{Prop}

\begin{proof}
Note $S_p^{[-h,h]}=\STRICT^{[1,h]}=\EXS{0}$, and note $a_d(\lambda)\leq 0$ for $0\leq d\leq h$, $\lambda\in\EXS{0}$. 
Let $\lambda\in \STRICT_p^{[jp-h,jp+h]}$ for $j\geq 1$.
Notice that $\SFOURC,\SFIVEC$ for $\lambda$ are vacuously true, (S2) is equivalent to $\ell(\lambda)\leq h$, and $(a_{d-1}(\lambda),a_{d}(\lambda))=(0,1)$ violates $\STHREE{j}{d}$.
\end{proof}

\begin{Def}\label{invdef}
For $\mu_+\in S_p^{[jp+1,jp+h]}$ and $\mu_-\in S_p^{[jp-h,jp-1]}$, we put 
\begin{align*}
b_i=b_i(\mu_+,\mu_-) &:= -i+\sum_{k=jp-i}^{jp-1}m_k(\mu_-)+\sum_{k=jp+1}^{jp+i}m_k(\mu_+)\quad\textrm{for\quad$0\leq i\leq h$},\\
j_0=j_0(\mu_+,\mu_-) &:= \min \{0\leq i\leq h\mid b_i=b \},
\quad\textrm{where $b=b(\mu_+,\mu_-):=\max\{b_0,\dots,b_h\}$},
\end{align*}
and define $\INVMAP{j}(\mu_+,\mu_-)\in \STRICT_p^{[jp-h,jp+h]}$ by
\begin{align*}
\INVMAP{j}(\mu_+,\mu_-)_{jp\pm i}=\begin{cases}
b & \textrm{if $i=0$}, \\
1-m_{jp\mp i}(\mu_{\mp}) & \textrm{if $1\leq i\leq j_0$}, \\
m_{jp\pm i}(\mu_{\pm}) & \textrm{if $j_0<i\leq h$}.
\end{cases}
\end{align*}
\end{Def}

There is an obvious operational description for the construction of $\INVMAP{j}(\mu_+,\mu_-)$, similar to the map $\lambda\mapsto(\lambda^+,\lambda^-)$ 
as in Remark \ref{consteng}. 

\begin{Ex}\label{pict}
  Let $p=33$, $j=1$, and take
\begin{align*}
  \mu_+=(47,46,44,41,40,37,36),\quad
  \mu_-=(31,30,29,27,26,25,23,21,20,19,18).
  \end{align*}
The picture below gives $j_0=14$, $b=3$ and
\begin{align*}
  \INVMAP{1}(\mu_+,\mu_-)=(44,42,38,34,33,33,33,32,31,28,27,24,23,21,18).
\end{align*}
\begin{align*}
  \begin{array}{l|ccccccccccccccccc}
    i & 16 & 15 & 14 & 13 & 12 & 11 & 10 & 9 & 8 & 7 & 6 & 5 & 4 & 3 & 2 & 1 & 0 \\ \hline
    b_i & 2 & 3 & 3 & 2 & 1 & 1 & 1 & 1 & 2 & 1 & 0 & 0 & 1 & 0 & -1 & -1 & 0 \\ \hline
    \mu_+ & \circ & \circ & 47 & 46 & \circ & 44 & \circ & \circ & 41 & 40 & \circ & \circ & 37 & 36 & \circ & \circ \\
    \mu_- & \circ & 18 & 19 & 20 & 21 & \circ & 23 & \circ & 25 & 26 & 27 & \circ & 29 & 30 & 31 & \circ \\ \hline
    \phantom{\mu_+} & \circ & \circ & \circ & \circ & \circ & 44 & \circ & 42 & \circ & \circ & \circ & 38 & \circ & \circ & \circ & 34 \\
    \phantom{\mu_-} & \circ & 18 & \circ & \circ & 21 & \circ & 23 & 24 & \circ & \circ & 27 & 28 & \circ & \circ & 31 & 32  
\end{array}
\end{align*}
\end{Ex}

\begin{Thm}\label{pmbij}
The correspondence (see Definition \ref{magicmap})
\begin{align*}
S_p^{[jp-h,jp+h]}\LONGTO S_p^{[jp+1,jp+h]}\times S_p^{[jp-h,jp-1]},\quad
\lambda\LONGMAPSTO (\lambda^+,\lambda^-),
\end{align*}
is a bijection with the inverse $\INVMAP{j}:S_p^{[jp+1,jp+h]}\times S_p^{[jp-h,jp-1]}\ISOM S_p^{[jp-h,jp+h]}$.
\end{Thm}

\begin{proof} On the composite of the operations in Definition \ref{magicmap}
  and Definition \ref{invdef}, it is easy to see, for
  $\lambda\in \EXS{j}$ and 
  $\mu_+\in S_p^{[jp+1,jp+h]}$, $\mu_-\in S_p^{[jp-h,jp-1]}$, we have
\begin{align}
b_i(\lambda^+,\lambda^-) &=
\begin{cases}
a_0(\lambda)-a_i(\lambda) & \textrm{if $0\leq i<i_0(\lambda)$}, \\
a_0(\lambda) & \textrm{if $i=i_0(\lambda)$}, \\
a_0(\lambda)+a_i(\lambda) & \textrm{if $i_0(\lambda)<i\leq h$},
\end{cases}
\label{ba123}
\\
  a_i(\INVMAP{j}(\mu_+,\mu_-)) &=
\begin{cases}
b(\mu_+,\mu_-)-b_i(\mu_+,\mu_-) & \textrm{if $0\leq i<j_0(\mu_+,\mu_-)$}, \\
0 & \textrm{if $i=j_0(\mu_+,\mu_-)$}, \\
b_i(\mu_+,\mu_-)-b(\mu_+,\mu_-) & \textrm{if $j_0(\mu_+,\mu_-)<i\leq h$}.
\end{cases}
\label{ab123}
\end{align}

First, we prove $\nu=\INVMAP{j}(\mu_+,\mu_-)$ for $\mu_+\in S_p^{[jp+1,jp+h]}, \mu_-\in S_p^{[jp-h,jp-1]}$ belongs to $S_p^{[jp-h,jp+h]}$. 
By \eqref{ab123}, we have $a_h(\nu)=\ell(\nu)-h\leq 0$. 
If we have $a_d(\nu)=1$ (resp. $a_{d-1}(\nu)=0$) for $1\leq d\leq h$, then we must have $d<j_0(\mu_+,\mu_-)$ (resp. $d-1\geq j_0(\mu_+,\mu_-)$) by \eqref{ab123}. Thus $\nu\in S_p^{[jp-h,jp+h]}$ by Proposition \ref{charSp}.

To prove $\mu_\pm=\nu^\pm$ for $\nu=\INVMAP{j}(\mu_+,\mu_-)\in S_p^{[jp-h,jp+h]}$, where $\mu_+\in S_p^{[jp+1,jp+h]}, \mu_-\in S_p^{[jp-h,jp-1]}$, it is enough to prove $i_0(\nu)=j_0(\mu_+,\mu_-)$. This follows from \eqref{ab123}.

To prove $\lambda=\INVMAP{j}(\lambda^+,\lambda^-)$ for $\lambda\in S_p^{[jp-h,jp+h]}$, 
it is enough to prove $j_0(\lambda^+,\lambda^-)=i_0(\lambda)$ and $b(\lambda^+,\lambda^-)=a_0(\lambda)(=m_{jp}(\lambda))$.
For this purpose, by \eqref{ba123}, we just need to see $a_i(\lambda)\leq 0$ for $i_0(\lambda)<i\leq h$.
Assume the contrary, and take the minimum instance $i_0(\lambda)<d\leq h$ with $a_d(\lambda)>0$. Then, we must
have $a_d(\lambda)=1$ and $a_{d-1}(\lambda)=0$ by \eqref{aexreason}. Thanks to Proposition \ref{charSp}, this contradicts $\lambda\in S_p^{[jp-h,jp+h]}$.
\end{proof}

In the latter half of the proof of Theorem \ref{pmbij}, we proved the following.

\begin{Lem}\label{aproperty}
For $\lambda\in S_p^{[jp-h,jp+h]}$, we have $a_i(\lambda)\leq 0$ whenever $i_0(\lambda)\leq i\leq h$.
\end{Lem}

\subsection{Zigzag property}\label{zzsec}
The rest of this section is devoted to proving Theorem \ref{zigzagplus}, which plays a key role in the proof of Theorem \ref{maintheorem}.
We only use Corollary \ref{ketsuron} after this section, and
readers can safely skip the proof, if they accept it.

\begin{Thm}\label{zigzagplus}
For $j\geq 0$ and take $\lambda\in S_p^{[jp-h,jp+h]}$ and $\mu\in S_p^{[(j+1)p-h,(j+1)p+h]}$.
Then, we have $\CONCAT{\mu}{\lambda}\in S_p$ if and only if $\CONCAT{\mu}{\lambda^+}\in S_p$.
\end{Thm}

\begin{Cor}\label{zigzagminus}
For $j\geq 1$ and take $\lambda\in S_p^{[jp-h,jp+h]}$ and $\nu\in S_p^{[(j-1)p-h,(j-1)p+h]}$.
Then, we have $\CONCAT{\lambda}{\nu}\in S_p$ if and only if $\CONCAT{\lambda^-}{\nu}\in S_p$.
\end{Cor}

\begin{proof}
Define $\lambda'$ by
$m_i(\lambda')=m_{2jp-i}(\lambda)$.
We easily see $i_0(\lambda)=i_0(\lambda')$ and $m_i((\lambda')^{\pm})=m_{2jp-i}(\lambda^{\mp})$. 
The rest is routine applications of Lemma \ref{invweakideal2} and Theorem \ref{zigzagplus}.
\end{proof}

\begin{Cor}\label{zigzagtotal}
For $j\geq 1$ and take $\lambda\in S_p^{[jp-h,jp+h]}$ and $\nu\in S_p^{[(j-1)p-h,(j-1)p+h]}$.
Then, we have $\CONCAT{\lambda}{\nu}\in S_p$ if and only if $\CONCAT{\lambda^-}{\nu^+}\in S_p$.
\end{Cor}

\begin{Cor}\label{ketsuron}
The following map is well-defined, and is a bijection.
\begin{align*}
  S_p &\longrightarrow
\SEIGEN_{j\geq 1} S_p^{[(j-1)p+1,jp-1]},\\
\lambda &\LONGMAPSTO
(\CONCAT{(\lambda|_{[jp-h,jp+h]})^{-}}{(\lambda_{[(j-1)p-h,(j-1)p+h]})^+})_{j\geq 1}.
\end{align*}
\end{Cor}

\begin{proof}
  Apply
  Lemma \ref{weakideal},
  Lemma \ref{invweakideal},
  Theorem \ref{pmbij} and
  Corollary \ref{zigzagtotal},
  noting $\nu=\INVMAP{0}(\nu,\varnothing)$ for $\nu\in S_p^{[1,h]}$.
  On the notations, recall \S\ref{noco}.
\end{proof}

After preparations, we prove Theorem \ref{zigzagplus} in \S\ref{zigzagproofb} and \S\ref{zigzagproofa}.


\begin{Prop}\label{directcalc}
For $j\geq 0$ and $\lambda\in S_p^{[jp-h,jp+h]}$, we have 
\begin{align*}
\sum_{k=c}^{h}m_{jp+k}(\lambda^+)=
\begin{cases}
\Delta^{+}+a_{c-1}(\lambda)=\Delta^{-}-(c-1) & \textrm{if $1\leq c\leq i_0(\lambda)$}, \\
\Delta^{+}=\Delta^{-}-(c-1+a_{c-1}(\lambda)) & \textrm{if $i_0(\lambda)<c\leq h$},
\end{cases}
\end{align*}
for $1\leq c\leq h$, where $\Delta^{+}=\sum_{k=c}^{h}m_{jp+k}(\lambda)$ and $\Delta^{-}=\sum_{k=1-c}^{h}m_{jp+k}(\lambda)$.
\end{Prop}

\begin{proof}
Direct calculations using \eqref{henkanm2} and \eqref{henkanm} in Definition \ref{magicmap}.
Write \mbox{$i_0=i_0(\lambda)$}.

Assume $i_0<c\leq h$. Because $m_{jp+i}(\lambda^+)=m_{jp+i}(\lambda)$ for $i_0<i\leq h$, $\sum_{k=c}^{h}m_{jp+k}(\lambda^+)=\Delta^{+}$ is obvious. Similarly, $\Delta^{-}-(c-1+a_{c-1}(\lambda))
=\sum_{k=1-c}^{h}m_{jp+k}(\lambda)-\sum_{k=-(c-1)}^{c-1}m_{jp+k}(\lambda)=\sum_{k=c}^{h}m_{jp+k}(\lambda)=\sum_{k=c}^{h}m_{jp+k}(\lambda^+)$.

Assume $1\leq c\leq i_0(\lambda)$. By \eqref{henkanm},
$\sum_{k=c}^{h}m_{jp+k}(\lambda^+)=X$, where $X=\sum_{k=i_0+1}^{h}m_{jp+k}(\lambda)+\sum_{k=c}^{i_0}(1-m_{jp-k}(\lambda))$. Note $X=\Delta^{+}-\sum_{k=c}^{i_0}(-1+m_{jp-k}(\lambda)
+m_{jp+k}(\lambda))=\Delta^{+}-(a_{i_0}(\lambda)-a_{c-1}(\lambda))$.
Similarly, 
$X=\Delta^{-}-\sum_{k=1-c}^{i_0}m_{jp+k}(\lambda)+\sum_{k=c}^{i_0}(1-m_{jp-k}(\lambda))
=\Delta^{-}-\sum_{k=-i_0}^{i_0}m_{jp+k}(\lambda)+(i_0-c+1)=\Delta^{-}-(a_{i_0}(\lambda)+i_0)+(i_0-c+1)$.
\end{proof}





\subsection{Proof of $\CONCAT{\mu}{\lambda^+}\not\in S_p$ implies $\CONCAT{\mu}{\lambda}\not\in S_p$}\label{zigzagproofb}
Let $\nu=\CONCAT{\mu}{\lambda}$ and $\nu'=\CONCAT{\mu}{\lambda^+}$. Assume $\nu'\not\in S_p$. Then, there exist $1\leq a\leq \ell(\mu)$ and $1\leq b\leq \ell(\lambda^+)$ such that (X) or (Y). Here, $D'=(\ell(\mu)-a)+(b-1)$ is the number of parts in $\nu'$ between $\mu_a$ and $(\lambda^+)_b$ (not including them). 
\begin{enumerate}
\item[(X)]\label{zigzagaAA} $\mu_a-(\lambda^+)_b<p$ with $D'=h-1$.
\item[(Y)]\label{zigzagaBB} $\mu_a=jp+(h+d),(\lambda^+)_b=jp+h+1-d$ for some $1\leq d\leq h$ with $D'=d-1$.
\end{enumerate}
We write $(\lambda^+)_b=jp+c$ for some $1\leq c\leq h$.
Clearly, $\sum_{k=jp+c}^{jp+h}m_k(\lambda^+)=b\geq 1$.

The case $i_0(\lambda)<c$ is trivial because of $\lambda_b=(\lambda^+)_b$.
In the rest, we assume $1\leq c\leq i_0(\lambda)$.
By Proposition \ref{directcalc}, we have $\sum_{k=jp+1-c}^{jp+h}m_k(\lambda)=b+(c-1)$ and
$\lambda_{b+c-1}\geq jp+1-c$.
In each case, we see $\nu\not\in S_p$ as follows.
Here, $D=D'+(c-1)$ is the number of parts in $\nu$ between $\mu_a$ and $\lambda_{b+c-1}$. 
\begin{enumerate}
\item[(X)] $\mu_a\leq jp+p+(c-1)$ and $\lambda_{b+c-1}\geq jp-(c-1)$ with $D=(h-1)+(c-1)=h+c-2$ contradicts Proposition \ref{indnum} (\ref{indnumb}).
\item[(Y)] $\mu_a=jp+(h+d)$ and $\lambda_{b+h-d}\geq jp-(h-d)$ with $D=(d-1)+(h-d)=h-1$ violates (S2).
\end{enumerate}

\subsection{Proof of $\CONCAT{\mu}{\lambda}\not\in S_p$ implies $\CONCAT{\mu}{\lambda^+}\not\in S_p$}\label{zigzagproofa}
\subsubsection{Setting}
Let $\nu=\CONCAT{\mu}{\lambda}$ and $\nu'=\CONCAT{\mu}{\lambda^+}$.
Assume 
$\nu\not\in S_p$ in spite of $\lambda\in S_p^{[jp-h,jp+h]}$, $\mu\in S_p^{[(j+1)p-h,(j+1)p+h]}$. Then, there exist $1\leq a\leq \ell(\mu)$ and $1\leq b\leq \ell(\lambda)$ such that (P) or (Q) or (R). Here, $D=(\ell(\mu)-a)+(b-1)$ is the number of parts in $\nu$ between $\mu_a$ and $\lambda_b$. 
\begin{enumerate}
\item[(P)]\label{zigzagaA} $\mu_a-\lambda_b<p$ with $D=h-1$.
\item[(Q)]\label{zigzagaB} $\mu_a=jp+(h+d),\lambda_b=jp+h+1-d$ for some $1\leq d\leq h$ with $D=d-1$.
\item[(R)]\label{zigzagaC} $\mu_a=jp+(p+d),\lambda_b=jp-d$ for some $0\leq d\leq h$ with $D=h-1+d$.
\end{enumerate}

\subsubsection{The case $\lambda_b=jp+c$ for $1\leq c\leq h$}
This case occurs when (P) or (Q). Clearly, $\sum_{k=jp+c}^{jp+h}m_k(\lambda)= b(\geq 1)$.
By Proposition \ref{directcalc}, we have $\sum_{k=jp+c}^{jp+h}m_k(\lambda^+)\geq b$. Thus, we get
$1\leq b\leq \ell(\lambda^+)$ and $(\lambda^+)_b\geq \lambda_b$. In each case, we see $\nu'\not\in S_p$ as follows.
\begin{enumerate}
\item[(P)] $\mu_a-(\lambda^+)_b<p$ with $D=h-1$ violates (S2) again.
\item[(Q)] $\mu_a=jp+(h+d)$ and $(\lambda^+)_b\geq jp+h+1-d$ with $D=d-1$ contradicts Proposition \ref{indnum} (\ref{indnuma}).
\end{enumerate}

\subsubsection{The case $\lambda_b=jp-c$ for $0\leq c\leq h$}
It occurs when (P) or (R). Clearly, $\sum_{k=jp-c}^{jp+h}m_k(\lambda)\geq b(\geq 1)$. We may assume $c\ne h$ in each case as follows.
\begin{enumerate}
\item[(P)] If $\lambda_b=jp-h$, then $\mu_a-\lambda_b\geq p$ for any $jp+h+1\leq \mu_a(\leq jp+p+h)$.
\item[(R)] As in Remark \ref{redudef}, $\SFIVE{j}{h}$ is redundant.
\end{enumerate}
Let $c'=c+1$. Then $1\leq c'\leq h$, and we have $\sum_{k=jp+c'}^{jp+h}m_k(\lambda^+)\geq b-(c'-1)=b-c$
by Proposition \ref{directcalc} and Lemma \ref{aproperty}. We see $b>c$ in each case as follows.
\begin{enumerate}
\item[(P)] Since $\mu_a\leq jp+p-(1+c)$ and $\mu_{\ell(\mu)}\geq jp+p-h$, we have $\ell(\mu)-a\leq h-(c+1)$. Thus, $b-1=D-(\ell(\mu)-a)\geq c$.
\item[(R)] By Proposition \ref{charSp}, $\ell(\mu)\leq h$. Thus, $b-1\geq D-(h-1)=d(=c)$.
\end{enumerate}
Thus, we get $(\lambda^+)_{b-c}\geq jp+(c+1)$. In each case, we see $\nu'\not\in S_p$ as follows.
Here, $D'=D-c$ is the number of parts in $\nu'$ between $\mu_a$ and $(\lambda^+)_{b-c}$. 
\begin{enumerate}
\item[(P)] $\mu_a\leq jp+p-(1+c)=jp+h+(h-c)$ and $(\lambda^+)_{b-c}\geq jp+(c+1)=jp+h+1-(h-c)$ with $D'=h-1-c$ and $1\leq h-c\leq h$ contradicts Proposition \ref{indnum} (\ref{indnuma}).
\item[(R)] $\mu_a=jp+(p+d)$ and $(\lambda^+)_{b-d}\geq jp+(d+1)$ with $D'=h-1$ violates (S2).
\end{enumerate}

\section{Precrystal structures on $S_p$}\label{pkcsec}

\subsection{The category of precrystals}\label{pcrcat}
Let $\PFN$ (resp. $\SETS$) be the category of sets and partial functions
(resp. sets with a basepoint and basepoint-preserving maps).
A functor $\CONC:\PFN\LONGTO\SETS$ defined for a morphism $f\in\HOM_{\PFN}(A,B)$ by
\begin{align*}
\CONC(f:A\to B):(A\sqcup\{\ZERO\},\ZERO) &\LONGTO (B\sqcup\{\ZERO\},\ZERO),\\
a &\LONGMAPSTO \begin{cases}
f(a) & \textrm{if $a\in A$ and $f(a)$ is defined},\\
\ZERO & \textrm{if otherwise},
\end{cases}
\end{align*}
affords a category equivalence $\CONC:\PFN\ISOM\SETS$~\cite[Exercise 2.3.12]{Lei}. 
We denote a functor $\CONC|_{\SET}:\SET\to\SETS$ by $\CONCC$, regarding $\SET$ as a subcategory of $\PFN$.

\begin{Def}[{\cite[p.500]{KKMMNN}}]
For a nonempty set $I$, an $I$-precrystal structure on a set $B$ is a data of Kashiwara operators $\KE_i, \KF_i\in\HOM_{\PFN}(B,B)$ for $i\in I$ which satisfies the two conditions below.
\begin{enumerate}
\item For any $b\in B$, both of the following take finite values.
\begin{align*}
\varepsilon_i(b) &= \sup\{n\geq 0\mid (\CONC\KE_i)^n(b)\ne\ZERO\}, \\
\varphi_i(b) &= \sup\{n\geq 0\mid (\CONC\KF_i)^n(b)\ne\ZERO\}.
\end{align*}
\item 
  We have $\KE_i b=b'$ if and only if $b=\KF_i b'$
  for $b, b'\in B$.
\end{enumerate}
\end{Def}

\begin{Rem}\label{crgr}
A crystal graph $\CG(B)$ of an $I$-precrystal $B$ is an $I$-colored directed graph whose vertex set is $B$, and
adjacent relations $\AR{b}{i}{b'}$ (meaning that there is an $i$-colored directed edge from $b$ to $b'$) if and only if $\KF_ib=b'$, where $i\in I$ and $b,b'\in B$.
Obviously, $I$-precrystal structure can be recovered from the crystal graph.
\end{Rem}

We usually omit the index set $I$ when it is clear from the context.
Precrystals form a category by declaring that a morphism between $(B, (\KE_i)_{i\in I}, (\KF_i)_{i\in I})$ and $(B', (\KEP{i})_{i\in I}, (\KFP{i})_{i\in I})$ is a map $g:B\to B'$ such that, for $i\in I$, we have
\begin{align*}
\CONC(\KEP{i})\circ \CONCC(g)=\CONCC(g)\circ \CONC(\KE_i),\quad
\CONC(\KFP{i})\circ \CONCC(g)=\CONCC(g)\circ \CONC(\KF_i).
\end{align*}

\subsection{The tensor product rule}\label{tensec}
The category of $I$-precrystals 
is a strict monoidal category with a unit precrystal $\{\bullet\}$, where $\CONC(\KE_i)(\bullet)=\ZERO=\CONC(\KF_i)(\bullet)$ for $i\in I$.

\begin{Def}[{\cite[Definition 2.3.2]{Kas}}]\label{crystaltensor}
Given two precrystals $(B_1,(\KEO{i})_{i\in I}, (\KFO{i})_{i\in I})$ and $(B_2,(\KET{i})_{i\in I}, (\KFT{i})_{i\in I})$, we denote by $B_1\otimes B_2=B_1\times B_2$ the precrystal whose Kashiwara operators are defined according to the tensor product rule
\begin{align*}
\CONC(\KE_i)(b_1\otimes b_2) &= \begin{cases}
(\CONC(\KEO{i})b_1)\otimes b_2 & \textrm{if $\varphi_i(b_1)\geq \varepsilon_i(b_2)$}, \\
b_1\otimes (\CONC(\KET{i})b_2) & \textrm{if $\varphi_i(b_1)<\varepsilon_i(b_2)$},
\end{cases}\\
\CONC(\KF_i)(b_1\otimes b_2) &= \begin{cases}
(\CONC(\KFO{i})b_1)\otimes b_2 & \textrm{if $\varphi_i(b_1)>\varepsilon_i(b_2)$}, \\
b_1\otimes (\CONC(\KFT{i})b_2) & \textrm{if $\varphi_i(b_1)\leq\varepsilon_i(b_2)$}.
\end{cases}
\end{align*}
Here, $b_1\otimes b_2$ means $(b_1,b_2)\in B_1\times B_2$, and $\ZERO=\ZERO\otimes b_2=b_1\otimes\ZERO$ for $b_1\in B_1,b_2\in B_2$.
\end{Def}



\begin{Rem}
For precrystals $B_1,\dots,B_m$ and $b=b_1\otimes\cdots\otimes b_m\in B_1\otimes\cdots\otimes B_m$, we can see that (see ~\cite[\S2.1]{KN} and ~\cite[\S2.3]{Kas})
\begin{align}
\begin{split}
\varepsilon_i(b) &= \max \{\varepsilon_{i,k}:=\varepsilon_i(b_k)-\sum_{1\leq\mu<k}(\varphi_i(b_{\mu})-\varepsilon_i(b_{\mu}))\mid 1\leq k\leq m\},\\
\varphi_i(b) &= \max \{\varphi_{i,k}:=\varphi_i(b_k)+\sum_{k<\mu\leq m}(\varphi_i(b_{\mu})-\varepsilon_i(b_{\mu}))\mid 1\leq k\leq m\},\\
\KE_ib &= b_1\otimes\cdots\otimes (\KE_ib_{k_e})\otimes\cdots \otimes b_m,\\
\KF_ib &= b_1\otimes\cdots\otimes (\KF_ib_{k_f})\otimes\cdots \otimes b_m,
\end{split}
\label{tensorformula}
\end{align}
where $k_e$ (resp. $k_f$) is the smallest (resp. largest) $1\leq k\leq m$ such that
$\varepsilon_i(b)=\varepsilon_{i,k}$, which
is equivalent to $\varphi_i(b)=\varphi_{i,k}$,
if $\varepsilon_i(b)>0$ (resp. $\varphi_i(b)>0$), and $\CONC(\KE_i)b=\ZERO$ (resp. $\CONC(\KF_i)b=\ZERO$) if otherwise.
\end{Rem}

\begin{Rem}\label{algtensor}
There is an algorithmic translation of the tensor product rule.
Given $b=b_1\otimes\cdots\otimes b_m\in B_1\otimes \cdots\otimes B_m$, write a $\pm$-sequence (the $i$-signature of $b$)
\begin{align*}
\underbrace{-\cdots-}_{\varepsilon_i(b_1)}\underbrace{+\cdots+}_{\varphi_i(b_1)}
\quad\cdots\quad
\underbrace{-\cdots-}_{\varepsilon_i(b_m)}\underbrace{+\cdots+}_{\varphi_i(b_m)}.
\end{align*}
Then, we obtain the reduced $i$-signature of $b$ by deleting every occurrence of a pair $+-$.
It is of the form $\underbrace{-\cdots-}_{\varepsilon_i(b)}\underbrace{+\cdots+}_{\varphi_i(b)}$, and
$k_e$ (resp. $k_f$) in \eqref{tensorformula} corresponds to the right (resp. left) most $-$ (resp. $+$) if any (see ~\cite[\S2.1]{KN}).
\end{Rem}

\begin{Rem}\label{inftensor}
For a sequence $(B_n,b_n)_{n\geq 0}$ of pairs of precrystal $B_n$ and an element $b_n\in B_n$ such that $\varphi_i(b_{n+1})=\varepsilon_i(b_n)$ for $n\geq 0$,
the set $\otimes'_{n\geq 0} (B_n,b_n)$ defined as
\begin{align*}
\{
(x_n)_{n\geq 0}\in {{\prod_{n\geq 0}}}B_n\mid
\textrm{$x_n=b_n$ for sufficiently large $n$}
\}
\end{align*}
has a precrystal structure
by the formula \eqref{tensorformula} (e.g., see ~\cite[\S7.2]{Kas}), where
we regard $(x_n)_{n\geq 0}$ as
$\cdots\otimes x_1\otimes x_0$.
\end{Rem}

\subsection{Precrystals structures on $S_{p}^{[(j-1)p+1,jp-1]}$}
Throughout the paper, we put
\begin{align}
\AI{p}=\{0,\dots,h\},\quad
\AII{p}=\{0,\dots,h-1\}
\label{aidex}
\end{align}
for an odd integer $p=2h+1\geq 3$.

\begin{Def}\label{pcryplus}
For $j\geq 0$, we put an $\AII{p}$-precrystal structure on $S_p^{[jp+1,jp+h]}$ by
\begin{align*}
  \begin{array}{lcl}
\KF_0\lambda =
\lambda\sqcup\{jp+1\}, & {} & \textrm{if $jp+1\not\in\lambda$}, \\[\jot]
\KF_i\lambda =
\lambda\setminus\{jp+i\}\sqcup\{jp+i+1\}, & {} & \textrm{if $jp+i\in\lambda$ and $jp+i+1\not\in\lambda$}, 
\end{array}
  \end{align*}
where $1\leq i\leq h-1$ and $\lambda\in S_p^{[jp+1,jp+h]}$.
\end{Def}

\begin{Def}\label{pcryminus}
For $j\geq 1$, we put an $\AII{p}$-precrystal structure on $S_p^{[jp-h,jp-1]}$ by
\begin{align*}
\begin{array}{lcl}
  \KF_0\lambda =
\lambda\setminus\{jp-1\}, & {} & \textrm{if $jp-1\in\lambda$}, \\[\jot]
\KF_i\lambda =
\lambda\setminus\{jp-(i+1)\}\sqcup\{jp-i\}, & {} & \textrm{if $jp-i\not\in\lambda$ and $jp-(i+1)\in\lambda$}, 
\end{array}
\end{align*}
where $1\leq i\leq h-1$ and $\lambda\in S_p^{[jp-h,jp-1]}$.
\end{Def}


\begin{Prop}\label{crystraii}
Let $j\geq 1$ and $\lambda\in S_p^{[(j-1)p+1,jp-1]}$, and
put $\mu=\RES{\lambda}{jp-h}{jp-1}$, $\nu=\RES{\lambda}{(j-1)p+1}{(j-1)p+h}$.
For $i\in\AII{p}$, if $\KF_i(\mu\otimes\nu)=\mu'\otimes\nu'$ for $\mu'\in S_p^{[jp-h,jp-1]}$, $\nu'\in S_p^{[(j-1)p+1,(j-1)p+h]}$ (i.e., $\CONC(\KF_i)(\mu\otimes\nu)\ne\ZERO$), 
then $\CONCAT{\mu'}{\nu'}\in S_p^{[(j-1)p+1,jp-1]}$.
\end{Prop}

\begin{proof}
  Put $\lambda'=\CONCAT{\mu'}{\nu'}$.
  We see $\lambda'\in S_p^{[(j-1)p+1,jp-1]}$ 
case by case as follows. 
Note
\begin{align}
  S_p^{[(j-1)p+1,jp-1]}=\{\lambda\in\STRICT^{[(j-1)p+1,jp-1]}\!\mid
  \ell(\lambda)\leq h, \SFOUR{j-1}{d} (1\leq d\leq h)\}.
\label{charS}
\end{align}

\noindent{$\bullet$} $i=0$ and $\varphi_0(\mu)=1$, $\varepsilon_0(\nu)=0$.
Then, $\mu'=\mu\setminus\{jp-1\}$, $\nu'=\nu$, and thus $\lambda'\in S_p^{[(j-1)p+1,jp-1]}$ by Lemma \ref{weakideal}.

\noindent{$\bullet$} $i=0$ and $\varphi_0(\mu)=0$, $\varphi_0(\nu)=1$.
Then, $\mu'=\mu$, $\nu'=\nu\sqcup\{(j-1)p+1\}$.
Since $jp-1\not\in\mu$, $\SFOUR{j-1}{h}$ holds for $\lambda'$.
Thus, it suffices to prove $\ell(\lambda')\leq h$. 
In fact, $\ell(\lambda')>h$ implies $\sum_{k=(j-1)p+2}^{jp-2}m_k(\lambda)\geq h$.
It violates Proposition \ref{indnum} (\ref{indnuma}).

\noindent{$\bullet$} $1\leq i<h$ and $\varphi_i(\mu)=1$, $\varepsilon_i(\nu)=0$. 
Then, $\mu'=\mu\setminus\{jp-(i+1)\}\sqcup\{jp-i\}$, $\nu'=\nu$.
It is enough to show $\SFOUR{j-1}{h+1-i}$ holds for $\lambda'$.
Assume the contrary, i.e., $\sum_{k=(j-1)p+(i+1)}^{jp-(i+1)}m_k(\lambda')=h-i$. 
Thus, $\sum_{k=(j-1)p+(i+1)}^{jp-(i+1)}m_k(\lambda)=h+1-i$ which contradicts Proposition \ref{indnum} (\ref{indnuma}).

\noindent{$\bullet$} $1\leq i<h$ and $\varphi_i(\mu)=0$, $\varphi_i(\nu)=1$. 
Then, $\mu'=\mu$, $\nu'=\nu\setminus\{(j-1)p+i\}\sqcup\{(j-1)p+(i+1)\}$.
It is enough to show $\SFOUR{j-1}{h-i}$ holds for $\lambda'$.
Assume the contrary, i.e., $\sum_{k=(j-1)p+(i+2)}^{jp-(i+2)}m_{k}(\lambda')=h-1-i$ and $m_{jp-(i+1)}(\mu)=1$.
Note that $m_{jp-i}(\mu)=1$ follows from $\varphi_i(\mu)=0$. Thus, we have $\sum_{k=(j-1)p+(i+1)}^{jp-(i+1)}m_{k}(\lambda)=h-i$. It means $\SFOUR{j-1}{h+1-i}$ is violated for $\lambda$.
\end{proof}


\begin{Cor}\label{pcrysconcat}
For $j\geq 1$, we have a unique $\AI{p}$-precrystal structure on $S_{p}^{[(j-1)p+1,jp-1]}$ such that the obvious injection 
\begin{align*}
S_{p}^{[(j-1)p+1,jp-1]} &\LONGHOOKRIGHTARROW S_p^{[jp-h,jp-1]}\otimes S_p^{[(j-1)p+1,(j-1)p+h]},\\
\lambda &\LONGMAPSTO \RES{\lambda}{jp-h}{jp-1}\otimes\RES{\lambda}{(j-1)p+1}{(j-1)p+h},
\end{align*}
is an $\AII{p}$-precrystal morphism and $\KF_h$ is given by
\begin{align*}
\KF_h\lambda =
\lambda\setminus\{(j-1)p+h\}\sqcup\{jp-h\},\quad\textrm{if $(j-1)p+h\in\lambda$ and $jp-h\not\in\lambda$}.
\end{align*}
\end{Cor}

\begin{proof}
By Proposition \ref{crystraii} and its $\KE_i$ version, for which we do not duplicate a proof,
we just need to show
$\KF_h\lambda\in S_{p}^{[(j-1)p+1,jp-1]}$ (resp. $\KE_h\lambda\in S_{p}^{[(j-1)p+1,jp-1]}$)
for $\lambda\in S_p^{[(j-1)p+1,jp-1]}$ with $(j-1)p+h\in\lambda$, $jp-h\not\in\lambda$ (resp. $(j-1)p+h\not\in\lambda$, $jp-h\in\lambda$).
For this purpose, it is enough to check that $\SFOUR{j-1}{1}$ holds for $\KF_h\lambda$ (resp. $\KE_h\lambda$), which is trivial.
\end{proof}

\subsection{A behavior of Kashiwara operators}
Recall that $p=2h+1\geq 3$ is an odd integer.
A box in a row of a Young diagram has a residue obeying a pattern
$\young(01\cdots h\cdots 1001\cdots)$.
In other words, a box located at
$(jp+(i+1))$-th (resp. $((j+1)p-i)$-th) column has residue $i$ 
for $0\leq i\leq h$ and $j\geq 0$.
\begin{Prop}\label{onebox}
Let $j\geq 1$, and assume $\KF_i(\lambda\otimes\mu)=\lambda'\otimes\mu'$
for $\lambda,\lambda'\in S_p^{[jp+1,jp+h]}$, $\mu,\mu'\in S_p^{[jp-h,jp-1]}$ and
$i\in \AII{p}$.
Then, $\INVMAP{j}(\lambda',\mu')$ is obtained by adding a residue $i$-box to
$\INVMAP{j}(\lambda,\mu)$.
\end{Prop}

\begin{proof}
We need to consider 4 cases.
\begin{enumerate}
\item\label{caseX} $i=0$ and $\varphi_0(\lambda)=1$, $\varepsilon_0(\mu)=0$, i.e., $jp+1\not\in\lambda$, $jp-1\in\mu$.
Then, $\lambda'=\lambda\sqcup\{jp+1\}$, $\mu'=\mu$.
\item\label{caseY} $i=0$ and $\varphi_0(\lambda)=0$, $\varphi_0(\mu)=1$, i.e., $jp+1\in\lambda$, $jp-1\in\mu$.
Then, $\lambda'=\lambda$, $\mu'=\mu\setminus\{jp-1\}$. 
\item\label{caseZ} $1\leq i<h$ and $\varphi_i(\lambda)=1$, $\varepsilon_i(\mu)=0$, 
i.e., $jp+i\in\lambda$, $jp+(i+1)\not\in\lambda$ and
$\NOT(jp-i\in\mu,jp-(i+1)\not\in\mu)$. 
Then, $\lambda'=\lambda\setminus\{jp+i\}\sqcup\{jp+i+1\}$, $\mu'=\mu$.
Here, for two conditions $A$ and $B$, $\NOT(A,B)$ means that we do not have them simultaneously.
\item\label{caseW} $1\leq i<h$ and $\varphi_i(\lambda)=0$, $\varphi_i(\mu)=1$, 
i.e., $\NOT(jp+i\in\lambda,jp+(i+1)\not\in\lambda)$ and 
$jp-i\not\in\mu$, $jp-(i+1)\in\mu$. Then, $\lambda'=\lambda$, $\mu'=\mu\setminus\{jp-(i+1)\}\sqcup\{jp-i\}$.
\end{enumerate}
In each case, we see the claim holds as in \S\ref{pfcaseX}, \S\ref{pfcaseY}, \S\ref{pfcaseZ}, \S\ref{pfcaseW}.
\end{proof}

In the following, we define
$b_i=b_i(\lambda,\mu)$, $b'_i=b_i(\lambda',\mu')$ for $0\leq i\leq h$ and $j_0=j_0(\lambda,\mu)$, $j'_0=j'_0(\lambda',\mu')$, $b=b_{j_0}$, $b'=b'_{j'_0}$, $\MM{k}=m_{jp+k}(\INVMAP{j}(\lambda,\mu))$, $\MMM{k}=m_{jp+k}(\INVMAP{j}(\lambda',\mu'))$ for $-h\leq k\leq h$ (see Definition \ref{invdef}).
It is enough to show that the difference between $(\MM{k})_{k=-h}^{h}$ and $(\MMM{k})_{k=-h}^{h}$ is
one of the following for some $L\geq 0$ (resp. $L=0$) in case (\ref{caseX}), (\ref{caseY})
(resp. case (\ref{caseZ}), (\ref{caseW})).
\begin{itemize}
\item[(A)] $(\MM{i},\MM{i+1})=(L+1,0)$, $(\MMM{i},\MMM{i+1})=(L,1)$, and
  $\MM{j}=\MMM{j}$ for $j\ne i,i+1$.
\item[(B)] $(\MM{-(i+1)},\MM{-i})=(1,L)$, $(\MMM{-(i+1)},\MMM{-i})=(0,L+1)$, and
  $\MM{j}=\MMM{j}$ for $j\ne -(i+1),-i$.
\end{itemize}

\subsubsection{Proof in case (\ref{caseX})}\label{pfcaseX}
Since $b'_k=b_k+1$ for $k\geq 1$ and $b_1=0$, we consider:
\begin{itemize}
\item $b=0$. Then, $b'=1$, $j_0=0$, $j'_0=1$.
\item $b\geq 1$. Then, $b'=b+1$, $j_0=j'_0\geq 2$.
\end{itemize}
In each case, $\MMM{0}=\MM{0}+1$, $\MM{1}=\MMM{1}=\MMM{-1}=0$, $\MM{-1}=1$. 
Thus, (B) occurs. 

\subsubsection{Proof in case (\ref{caseY})}\label{pfcaseY}
We have $b'_k=b_k-1$ for $k\geq 1$ and $b_1=1$. The rest is similar to \S\ref{pfcaseX}, and we see
that (A) occurs. 
Because it is straightforward, we omit a detail.

\subsubsection{Proof in case (\ref{caseZ})}\label{pfcaseZ}
Let $x=b_{i-1}=b'_{i-1}$.
We divide (\ref{caseZ}) into 3 cases as follows, and we see
that one of (A) or (B) occurs. 
\begin{itemize}
\item $jp-i\not\in\mu$, $jp-(i+1)\in\mu$. 
Note $\MM{i}=1$, $\MM{-i}=0$, $\MM{i+1}=0$, $\MM{-(i+1)}=1$. Since $b'_i=x-1$, $b_i=b_{i+1}=b'_{i+1}=x$, $b_k=b'_{k}$ for 
$k>i+1$, we consider:
\begin{itemize}
\item[\SPE] $j_0=j'_0<i$, $b=b'\geq x$. Then, $\MMM{\pm i}=0$, $\MMM{\pm(i+1)}=1$.
\item[\SPE] $j_0=j'_0>i+1$, $b=b'\geq x+1$. Then, $\MMM{\pm i}=1$, $\MMM{\pm(i+1)}=0$.
\end{itemize}
\item $jp-i\not\in\mu$, $jp-(i+1)\not\in\mu$.
Note $\MM{i}=1$, $\MM{-i}=0$, $\MMM{i+1}=1$, $\MMM{-(i+1)}=0$.
Since $b_i=x$, $b'_i=b_{i+1}=b'_{i+1}=x-1$, $b_k=b'_{k}$ for 
$k>i+1$, we consider:
\begin{itemize}
\item[\SPE] $j_0=j'_0<i$, $b=b'\geq x$. 
Then, $\MM{\pm(i+1)}=0$, $\MMM{\pm i}=0$.
\item[\SPE] $j_0=j'_0>i+1$, $b=b'\geq x+1$. 
Then, $\MM{\pm(i+1)}=1$, $\MMM{\pm i}=1$.
\end{itemize}
\item $jp-i\in\mu$, $jp-(i+1)\in\mu$.
Note $\MM{i+1}=0$, $\MM{-(i+1)}=1$, $\MMM{i}=0$, $\MMM{-i}=1$.
Since $b'_i=x$, $b_i=b_{i+1}=b'_{i+1}=x+1$, $b_k=b'_{k}$ for 
$k>i+1$, we consider:
\begin{itemize}
\item[\SPE] $j_0=j'_0<i$, $b=b'\geq x+1$. 
Then, $\MM{\pm i}=1$, $\MMM{\pm(i+1)}=1$.
\item[\SPE] $j_0=j'_0>i+1$, $b=b'\geq x+2$. 
Then, $\MM{\pm i}=0$, $\MMM{\pm(i+1)}=0$.
\item[\SPE] $b=b'=x+1$, $j_0=i$, $j'_0=i+1$. 
Then, $\MM{\pm i}=0$, $\MMM{\pm(i+1)}=0$.
\end{itemize}
\end{itemize}

\subsubsection{Proof in case (\ref{caseW})}\label{pfcaseW}
The argument is similar to \S\ref{pfcaseZ}, and we see
that one of (A) or (B) occurs. 
Because it is straightforward, we omit a detail.


\subsection{Precrystal structures via zigzag property}\label{obinj}
By Corollary \ref{ketsuron} and 
Corollary \ref{pcrysconcat},
$S_p$ has a precrystal structure
imported from $\PRODPERFCRY$.

\begin{Thm}\label{finonebox}
If $\KF_i\lambda$ is defined (i.e., $\CONC(\KF_i)\lambda\ne\ZERO$) for $\lambda\in S_p$ and $i\in\AI{p}$, then it is obtained by adding a residue $i$-box to $\lambda$.
\end{Thm}

\begin{proof}
For $i\in \AII{p}$, the statement is the same as Proposition \ref{onebox}.
Put $\lambda^{(j)}=\lambda|_{[jp-h,jp+h]}$ for $j\geq 0$. Note that $\lambda^{(0)}=(\lambda^{(0)})^+$.
Assume that in
\begin{align*}
\KF_h(\cdots\otimes((\lambda^{(2)})^-\otimes(\lambda^{(1)})^+)\otimes((\lambda^{(1)})^-\otimes\lambda^{(0)})),
\end{align*}
$\KF_h$ is applied to $(\lambda^{(j)})^-\otimes(\lambda^{(j-1)})^+$ for some $j\geq 1$. 
In particular, 
we have
$\KF_h((\lambda^{(j)})^-\otimes(\lambda^{(j-1)})^+)=\mu\otimes\nu$, 
where $\mu=(\lambda^{(j)})^-\sqcup\{jp-h\}$ and $\nu=(\lambda^{(j-1)})^+\setminus\{(j-1)p+h\}$.
To prove the statement for $i=h$, it is enough to show 
\begin{align}
\begin{split}
  \INVMAP{j}((\lambda^{(j)})^+,\mu) &= \INVMAP{j}((\lambda^{(j)})^+,(\lambda^{(j)})^-)\sqcup\{jp-h\}, \\
\INVMAP{j-1}(\nu,(\lambda^{(j-1)})^-) &= \INVMAP{j-1}((\lambda^{(j-1)})^+,(\lambda^{(j-1)})^-)\setminus\{(j-1)p+h\}.
\end{split}
\label{onebox3}
\end{align}
For this purpose, it is enough to show (see Definition \ref{invdef})
\begin{align}
j_0((\lambda^{(j)})^+,\mu) &= j_0((\lambda^{(j)})^+,(\lambda^{(j)})^-)<h,\label{onebox1} \\
j_0(\nu,(\lambda^{(j-1)})^-) &= j_0((\lambda^{(j-1)})^+,(\lambda^{(j-1)})^-)<h.\label{onebox2}
\end{align}
We may assume $j\geq 2$ when proving \eqref{onebox2}.
Now \eqref{onebox1} (resp. \eqref{onebox2}) follows from $jp+h\not\in(\lambda^{(j)})^+$ (resp. $(j-1)p-h\not\in(\lambda^{(j-1)})^-$), which is shown as follows.

Assume $jp+h\in(\lambda^{(j)})^+$ (resp. $(j-1)p-h\in(\lambda^{(j-1)})^-$). 
If $(j+1)p-h\not\in(\lambda^{(j+1)})^-$ (resp. $(j-2)p+h\not\in(\lambda^{(j-2)})^+$), then
it contradicts the assumption that $\KF_h$ is applied to $(\lambda^{(j)})^-\otimes(\lambda^{(j-1)})^+$ 
(see Remark \ref{algtensor}). 
If $(j+1)p-h\in(\lambda^{(j+1)})^-$ (resp. $(j-2)p+h\in(\lambda^{(j-2)})^+$), then it contradicts
$\CONCAT{(\lambda^{(j+1)})^-}{(\lambda^{(j)})^+}\in S_p$ (resp. $\CONCAT{(\lambda^{(j-1)})^-}{(\lambda^{(j-2)})^+}\in S_p$) 
because $\SFOUR{j}{1}$ (resp. $\SFOUR{j-2}{1}$) is violated.
\end{proof}


\section{Kashiwara crystal theory}\label{kcsec}
We recall the necessary Lie theory for proof of Theorem \ref{maintheorem2}, yet basically, we assume that readers are familiar with Kac-Moody Lie algebras, quantum groups and Kashiwara crystal theory (\cite{Kac},~\cite{Kas}, ~\cite{Lus} are standard references).

\subsection{Kashiwara crystals}\label{kacr}
For a generalized Cartan matrix (GCM, for short) $A$ (see ~\cite[\S1.1]{Kac}), we fix a Cartan data, which is a 4-tuple $(P,P^{\CHECK},\Pi,\Pi^{\CHECK})$ such that
\begin{enumerate}
\item $P$ is a free $\Z$-module of rank $2|I|-\RANK A$, and $P^{\CHECK}=\HOM_{\Z}(P,\Z)$,
\item $\Pi=\{\alpha_i\mid i\in I\}$ are $\Z$-linearly independent elements in $P$,
\item $\Pi^{\CHECK}=\{h_i\mid i\in I\}$ are $\Z$-linearly independent elements in $P^{\CHECK}$,
\item $a_{ij}=\alpha_j(h_i)$ for $i,j\in I$.
\end{enumerate}

We denote by $\langle,\rangle:P^{\CHECK}\times P\to\Z$ the canonical pairing, and
denote by $P^+ = \{\lambda\in P\mid \textrm{$\langle h_i,\lambda\rangle\geq0$ for $i\in I$}\}$ the set of dominant integral weights.

\begin{Def}[{\cite[\S4.2]{Kas}}]\label{defkcr}
Let $A=(a_{ij})_{i,j\in I}$ be a GCM. 
An $A$-crystal is a $P$-weighted $I$-precrystal $B$, i.e., it is equipped with a map 
$\WT:B\to P$
such that
\begin{enumerate}
\item $\varphi_i(b)-\varepsilon_i(b)=\langle h_i,\lambda\rangle$ for $b\in B_{\lambda}:=\{b\in B\mid \WT(b)=\lambda\}$,
\item $\CONC(\KE_i)B_{\lambda}\subseteq B_{\lambda+\alpha_i}\sqcup\{\ZERO\}$ and $\CONC(\KF_i)B_{\lambda}\subseteq B_{\lambda-\alpha_i}\sqcup\{\ZERO\}$,
\end{enumerate}
where $i\in I$ and $\lambda\in P$.
\end{Def}

An $A$-crystal morphism $g:B\to B'$ is defined to be a precrystal 
morphism (see \S\ref{pcrcat}) such that $\WT(g(b))=\WT(b)$ for $b\in B$.
The tensor product precrystal $B_1\otimes B_2$ of $A$-crystals $B_1$ and $B_2$
is  an $A$-crystal by
$\WT(b_1\otimes b_2)=\WT(b_1)+\WT(b_2)$.

\subsection{Classical crystals}
Let $A=(a_{ij})_{i,j\in I}$ be an affine GCM as in ~\cite[Table Aff 1, Aff 2, Aff 3]{Kac}.
There exist unique $(a_i)_{i\in I},(a^{\CHECK}_i)_{i\in I}\in \mathbb{Z}^I_{\geq 1}$ characterized by $\GCD(a_i\mid i\in I)=1=\GCD(a^{\CHECK}_i\mid i\in I)$ and
$\sum_{i\in I}a^{\CHECK}_ia_{ij}=0=\sum_{i\in I}a_{ji}a_i$ for $j\in I$.

We take $\AIZ=0$ (resp. $\AIZ=n$)
if $A\ne A^{(2)}_{2n}$ (resp. $A=A^{(2)}_{2n})$
so that we always have $a_{\AIZ}=1$.
We also fix a Cartan data $(P,P^{\CHECK},\Pi,\Pi^{\CHECK})$ for $A$ as
\begin{align*}
P = \bigoplus_{i\in I}\Z\Lambda_i\oplus\Z\delta &\supseteq \Pi=\{\alpha_i=\sum_{j\in I}a_{ji}\Lambda_j+\delta_{i,\AIZ}\cdot\delta\},\\
P^{\CHECK} = \HOM_{\Z}(P,\Z) &\supseteq \Pi^{\CHECK}=\{h_i\mid i\in I\}\textrm{ by $\langle h_i,\Lambda_j\rangle=\delta_{ij}$ and $\langle h_i,\delta\rangle=0$}.
\end{align*}

Note that $P^+=\bigoplus_{i\in I}\Z_{\geq0}\Lambda_i\oplus\Z\delta\supseteq\Pi$ and $\delta=\sum_{i\in I}a_i\alpha_i$ in $P$.
We set $\PCL=P/\Z\delta$, and reserve $\CL:P\twoheadrightarrow\PCL$.
We also define the dual $\PCLC$ of $\PCL$ by
\begin{align*}
\PCLC=\HOM_{\Z}(\PCL,\Z)\cong\{h\in P^{\CHECK}\mid\langle h,\delta\rangle=0\}=\bigoplus_{i\in I}\Z h_i\subseteq P^{\CHECK}.
\end{align*}
Set $c=\sum_{i\in I}a^{\CHECK}_ih_i\in\PCLC\subseteq P^{\CHECK}$ so that $\langle c,\CL(\alpha_i)\RANGLE=0=\langle c,\alpha_i\rangle$ for $i\in I$, where
$\langle,\RANGLE:\PCLC\times\PCL\to\Z$ is the induced canonical pairing. 
Note that $\{h_i\mid i\in I\}\subseteq\PCLC$ are linearly independent, whereas $\{\CL(\alpha_i)\mid i\in I\}\subseteq\PCL$ are not.

\begin{Def}[{\cite[\S3.3]{KKMMNN2}}]
Let $A$ be an affine GCM. 
A classical $A$-crystal is defined by
replacing $P$ and $\langle,\rangle$ in Definition \ref{defkcr} with $\PCLL$ and $\langle,\RANGLE$, respectively.
\end{Def}

\subsection{Perfect crystals}\label{PERF}
When $A$ is affine, the highest weight crystal $B(\lambda)$ for $\lambda\in P^+$
can be realized as an infinite tensor product 
of level $\ell=\langle c,\lambda\rangle$ perfect crystal. 

\begin{Def}[{\cite[Definition 4.6.1]{KKMMNN2}}]
Let $A$ be an affine GCM and let $\ell\in\Z_{\geq 1}$.
We say that a classical $A$-crystal $B$ is a perfect crystal of level $\ell$ if
\begin{enumerate}
\item there exists a finite-dimensional integrable $U'_q$-module with a crystal base, which is isomorphic to $B$ (we omit an
  explanation in detail),
\item the crystal graph $\CG(B\otimes B)$ is connected (see Remark \ref{crgr}),
\item there exists $\lambda_0\in\PCLL$ such that $\WT(B)\subseteq \lambda_0+\sum_{i\ne \AIZ}\Z_{\leq 0}\CL(\alpha_i)$ and
$|B_{\lambda_0}|=1$,
\item $\langle c,\varepsilon(b)\RANGLE\geq \ell$ for $b\in B$,
\item the restrictions of $\varepsilon,\varphi:B\to\PCLL^+$ to 
$B_{\min}:=\{b\in B\mid \langle c,\varepsilon(b)\RANGLE=\ell\}$ are both 
  bijections to $(\PCLL^+)_{\ell}:=\{\lambda\in\PCLL^+\mid\langle c,\lambda\RANGLE=\ell\}$, where
\begin{align*}
\varepsilon(b)=\sum_{i\in I}\varepsilon_i(b)\CL(\Lambda_i),\quad
\varphi(b)=\sum_{i\in I}\varphi_i(b)\CL(\Lambda_i).
\end{align*}
\end{enumerate}
\end{Def}

It is conjectured (see ~\cite[Conjecture 2.1]{HKOTY} and ~\cite[Conjecture 2.1]{HKOTZ}) that 
\begin{enumerate}
\item 
Kirillov-Reshetikhin module $\KRM{i}{s}$ 
has a crystal base $\KRC{i}{s}$ (KR crystal, for short), where $i\in I\setminus\{\AIZ\}$ and $s\geq 1$,
\item $\KRC{i}{s}$ is perfect of level $s/t_i$ if $s/t_i\in\Z$, 
  where $t_i=\max(1,a_i/a^{\CHECK}_i)$. 
\end{enumerate}
They are settled affirmatively when $A$ is nonexceptional~\cite{FOS,OS}.
It is also expected~\cite{KNO} that any perfect crystal is a tensor product (see ~\cite{OSS}) of KR crystals.

In the paper, KR crystal $\KRC{n}{2}$ for $A=\AD{2n}$ (the Landlands dual of $A^{(2)}_{2n}$, see Figure \ref{twisted}) plays a crucial role. 
Recall our convention for $\AD{2n}$. In particular, 
\begin{align}
I=\{0,1,\dots,n\},\quad
\AIZ=0,\quad
t_s=\max(1,a_i/a^{\CHECK}_i)=1+\delta_{n,s},
\label{convdual}
\end{align}
where $s\in I\setminus\{\AIZ\}$.
The precrystal structure of $\KRC{n}{2}$ was determined by Jing-Misra-Okado~\cite[\S3]{JMO}, building on a work of Kashiwara-Nakashima~\cite[\S5]{KN}.

\begin{figure}
\[
\begin{array}{r@{\quad}l@{\qquad}l@{\quad}l}
A_2^{(2)} &\node{}{0} {\quad\!\!\!\!}\Llleftarrow{\quad\!\!\!\!\!} \node{}{1} &
A_{2n}^{(2)}  &\node{}{0}\Leftarrow \node{}{1}-\cdots-\node{}{{n-1}}\Leftarrow\node{}{{n}} \\
\AD{2} &\node{}{0} {\quad\!\!\!\!\!\!\!\!}\Rrrightarrow{\quad\!\!\!\!\!} \node{}{1} & 
\AD{2n}  &\node{}{0}\Rightarrow \node{}{1}-\cdots-\node{}{{n-1}}\Rightarrow\node{}{{n}} \\
D_\ell& \node{}{\alpha_1} - \node{}{\alpha_2} - \cdots - \node{\ver{}{\alpha_{\ell-1}}}{\alpha_{\ell-2}} -\node{}{\alpha_{\ell}} &
D_{n+1}^{(2)}  &\node{}{0}\Leftarrow \node{}{1}-\cdots-\node{}{n-1}\Rightarrow\node{}{n}
\end{array}
\]
\caption{The diagrams $A^{(2)}_{\textrm{even}}$, $\AD{\textrm{even}}$, $D_\ell$, $D^{(2)}_{n+1}$ ($n\geq 2, \ell\geq 4$)}
\label{twisted}
\end{figure}

\begin{Def}[{\cite[\S5.2, (5.3.2)]{KN}}]\label{fundcry}\label{knb}
For $n\geq 1$, we put $\KNB{n}{1}=\{\NODE{k}\mid -n\leq k\leq n\}$, 
and define its $(I\setminus\{\AIZ\})$-precrystal structure so that
\begin{align*}
\NODE{1}\stackrel{1}{\to}
\NODE{2}\stackrel{2}{\to}
\cdots\stackrel{n-2}{\to}
\NODE{n-1}\stackrel{n-1}{\to}
\NODE{n}\stackrel{n}{\to}
\NODE{0}\stackrel{n}{\to}
\NODE{-n}\stackrel{n-1}{\to}
\NODE{-(n-1)}\stackrel{n-2}{\to}
\cdots
\stackrel{2}{\to}
\NODE{-2}\stackrel{1}{\to}
\NODE{-1}
\end{align*}
is the crystal graph (see Remark \ref{crgr}).
For $1\leq s\leq n$, let
$\KNB{n}{s}$ is a subset of $(\KNB{n}{1})^{\otimes s}$
consists of $\NODE{i_1}\otimes\cdots\otimes\NODE{i_s}$ such that
\begin{enumerate}
\item $i_j\ORD i_{j+1}$ for $1\leq j<s$, and $i_j=i_{j+1}$ implies $i_j=0$, 
\item\label{knb2} $i_k=t,i_{\ell}=-t$ for $1\leq k<\ell\leq s$ and $1\leq t\leq n$ implies $\ell-k>s-t$.
\end{enumerate}
Here, $\ORD$ is a total order on $\{-n,\dots,n\}$ by $1\ORD\cdots\ORD n\ORD 0\ORD -n\ORD\cdots\ORD -1$.
\end{Def}

\begin{Prop}\label{KNsametime}
For $u=\NODE{i_1}\otimes\cdots\otimes\NODE{i_n}\in \KNB{n}{n}$, at least one of $\NODE{n}$, $\NODE{0}$, $\NODE{-n}$ is in $u$.
\end{Prop}

\begin{proof}
Assume the contrary, and we see $n\geq 2$ and that there exist $1\leq d<n$ and $1\leq k<\ell\leq n$ such that $i_k=d,i_{\ell}=-d$.
Take the largest $d$, and we have $\ell-k\leq n-d$ by $d<|i_{k+1}|,\dots,|i_{\ell-1}|<n$.
It contradicts Definition \ref{knb} (\ref{knb2}).
\end{proof}

\begin{Prop}[{\cite[Proposition 5.3.1]{KN}},{\cite[Theorem 3.2]{JMO}}]\label{KNJMO}
Let $1\leq s\leq n$.
\begin{enumerate}
\item\label{KNJMOa} There is a unique $(I\setminus\{\AIZ\})$-precrystal structure on $\KNB{n}{s}$ such that the obvious injection
$\KNB{n}{s}\LONGHOOKRIGHTARROW (\KNB{n}{1})^{\otimes s}$
is an $(I\setminus\{\AIZ\})$-precrystal morphism.
\item\label{KNJMOb} Moreover, $\KNB{n}{s}$ becomes a classical $\AD{2n}$-crystal such that $\KNB{n}{s}\cong \KRC{s}{t_s}$, after 
defining suitable weight $\WT:\KNB{n}{s}\to P_{\normalfont{\textrm{cl}}}$ and $\KE_0,\KF_0\in\HOM_{\PFN}(\KNB{n}{s},\KNB{n}{s})$ by
\begin{align*}
\KF_{0}\left(\NODE{i_1}\otimes\cdots\otimes\NODE{i_s}\right) =
\NODE{1}\otimes\NODE{i_1}\otimes\cdots\otimes\NODE{i_{s-1}},\quad\textrm{if $i_s=-1$},
\end{align*}
Note that we never have $(i_1,i_s)=(1,-1)$ by Definition \ref{knb} (\ref{knb2}).
\end{enumerate}
\end{Prop}

\begin{Ex}
The crystal graph $\CG(\KNB{n}{1})$ is obtained by adding a $0$-arrow from $\NODE{-1}$ to $\NODE{1}$
in Definition \ref{fundcry} (see ~\cite[p.481]{KKMMNN2}).
\end{Ex}


\begin{Thm}[{\cite[Theorem 4.5.2]{KKMMNN2}}]\label{kyotopath}
Let $B$ be a perfect crystal of level $\ell$ for an affine GCM $A$.
For $\lambda\in P^+$ with $\langle c,\lambda\rangle=\ell$,
let $\lambda_0=\CL(\lambda)$ and $\lambda_n=\varepsilon(\varphi^{-1}(\lambda_{n-1}))$ for $n\geq 1$.
There is 
a unique classical $A$-crystal isomorphism such that 
\begin{align}
  B(\lambda)
  \ISOM 
  \otimes'_{n\geq 0}(B,\varphi^{-1}(\lambda_n)),\quad
u_{\lambda}\LONGMAPSTO \cdots\otimes\varphi^{-1}(\lambda_1)\otimes \varphi^{-1}(\lambda_0),
\label{kyotopath2}
\end{align}
where $u_{\lambda}$ is the maximum element in $B(\lambda)$.
\end{Thm}


\begin{Cor}\label{kyotointhepaper}
There is a unique $\AD{2n}$-crystal isomorphism such that
\begin{align*}
B(\Lambda_n)\cong \otimes'_{k\geq 0}(\KNB{n}{s},\NODE{0}^{\otimes s}),\quad
u_{\Lambda_n}\LONGMAPSTO\cdots\otimes\NODE{0}^{\otimes s}\otimes\NODE{0}^{\otimes s}
\end{align*}
for any $1\leq s\leq n$.
\end{Cor}

\begin{proof}
  Because it is known that
  \eqref{kyotopath2} is lifted to an $A$-crystal isomorphism~\cite[Proposition 4.5.4]{KKMMNN2} by the energy function, it is enough to check $\varphi(\NODE{0}^{\otimes s})=\CL(\Lambda_n)=\varepsilon(\NODE{0}^{\otimes s})$, which follows from
Definition \ref{fundcry} and Remark \ref{algtensor}.
\end{proof}

\subsection{Proof of Theorem \ref{maintheorem2}}\label{pfsec}

By Corollary \ref{kyotointhepaper} for $s=n$ and Theorem \ref{finonebox},
it is enough to show 
Proposition \ref{finprop} (\ref{finprop3})
to prove Theorem \ref{maintheorem2}.

\begin{Prop}\label{finprop}
Let $h\geq 1$ and $p=2h+1$.
For $j\geq 1$, we define a bijection
\begin{align*}
r_j:\KNB{h}{1}\setminus\{\NODE{0}\} &\ISOM \{(j-1)p+1,\dots,jp-1\},\\
  \NODE{i} &\LONGMAPSTO
\begin{cases}
jp-(h+1-i) & \textrm{if $1\leq i\leq h$},\\
(j-1)p+(h+1+i) & \textrm{if $-h\leq i\leq-1$}.
\end{cases}
\end{align*}
\begin{enumerate}
\item\label{finprop1}
For $1\leq s\leq h$, 
  the map $\NONAME{h}{s}:\KNB{h}{s}\LONGTO S_p^{[(j-1)p+1,jp-1]}$ 
\begin{align*}
\NODE{i_1}\otimes\cdots\otimes\NODE{i_s}\LONGMAPSTO (r_j(i_{a-1}),\dots,r_j(i_1),r_j(i_s),\dots,r_j(i_{b+1}))
\end{align*}
is well-defined.
Here, $1\leq a\leq s+1$ (resp. $0\leq b\leq s$) is defined by
the condition $i_k>0\Leftrightarrow 1\leq k<a$ (resp. $i_k<0\Leftrightarrow b<k\leq s$). 
\item\label{covar} For $i\in \AI{p}\setminus\{h\}$ (see \eqref{aidex}), $\NONAME{h}{s}$ satisfies
\begin{align}
\begin{split}
\CONC(\KE_{h-i})\circ \CONCC(\NONAME{h}{s}) &= \CONCC(\NONAME{h}{s})\circ\CONC(\KE_{i}),\\
\CONC(\KF_{h-i})\circ\CONCC(\NONAME{h}{s}) &= \CONCC(\NONAME{h}{s})\circ\CONC(\KF_{i}).
\end{split}
\label{covarr}
\end{align}
\item\label{finprop3} For $s=h$, (\ref{covarr}) also holds for $i=h$, and $\NONAME{h}{h}$ is a bijection.
\end{enumerate}
\end{Prop}

\begin{proof}
We fix $u=\NODE{i_1}\otimes\cdots\otimes\NODE{i_s}\in \KNB{h}{s}$, and take $a$, $b$ as in (\ref{finprop1}). Note that we have $i_k=0$ if and only if $a\leq k\leq b$.

Recall \eqref{charS}.
To prove (\ref{finprop1}), put $\lambda=\NONAME{h}{s}(u)$, and
assume the contrary that there exist $1\leq x<y\leq \ell(\lambda)(\leq s)$ such that 
$\lambda_x=(j-1)p+(h+d)=jp-(h+1-d)$, $\lambda_y=(j-1)p+(h+1-d)$
for $d=y-x$. Note that $a\ne1$ and $b\ne s$.
Take $1\leq x'<a$ and $b<y'\leq s$ such that $r_j(i_{x'})=\lambda_x$ and $r_j(i_{y'})=\lambda_y$.
The assumption implies 
$i_{y'}=-d,i_{x'}=d$ with $y'-x'= s-d$, which contradicts Definition \ref{knb} (\ref{knb2}).

By Corollary \ref{pcrysconcat} and Proposition \ref{KNJMO} (\ref{KNJMOb}),
we have \eqref{covarr} for $i=0$.
To prove \eqref{covarr} for $i\in\AII{p}\setminus\{0,h\}$
(resp. $i=h$ and $s=h$ as in \eqref{finprop3}), we divide 16 (resp. $7=8-1$, see Proposition \ref{KNsametime}) cases
depending on each of $i$, $i+1$, $-(i+1)$, $-i$ (resp. $h$, $0$, $-h$)
is in $u$ or not.
Because it is straightforward
to check (\ref{covarr}) in each case by
Definition \ref{pcryplus},
Definition \ref{pcryminus},
Corollary \ref{pcrysconcat},
Definition \ref{fundcry},
Proposition \ref{KNJMO} (\ref{KNJMOa})
and
Remark \ref{algtensor},
we omit a detail.

To finish proving \eqref{finprop3}, it suffices to show that a map $S_p^{[(j-1)p+1,jp-1]}\to\KNB{h}{h}$ 
\begin{align*}
(\lambda_1,\dots,\lambda_{\ell}) \mapsto r_j^{-1}(\lambda_w)\otimes\cdots\otimes r_j^{-1}(\lambda_1)\otimes\NODE{0}^{\otimes(h-\ell)}
\otimes r_j^{-1}(\lambda_{\ell})\otimes\cdots\otimes r_j^{-1}(\lambda_{w+1})
\end{align*}
is well-defined, where $0\leq w\leq \ell$ is defined by the condition $jp-h\leq\lambda_i\leq jp-1\Leftrightarrow 1\leq i\leq w$.
The argument is similar to (\ref{finprop1}), and uses \eqref{charS} and Definition \ref{knb} (\ref{knb2}).
\end{proof}

\subsection{Proof of Theorem \ref{maintheorem}}\label{finalproof}
It is immediate
by virtue of Theorem \ref{maintheorem2} (see also ~\cite[(14.4.4)]{Kac}).


\section{Combinatorics of $\SCHUR_p$}\label{defsch}
\subsection{Definition of $\SCHUR_{p,\GOOD}$}\label{defsch1}
Let $p=2h+1\geq 3$ be an odd integer. 
Recall \S\ref{noco}. 
Assume that $\GOOD:S_p^{[p-h,p+h]}\hookrightarrow \STRICT^{[p-h,p+h]}$ is an injection such that 
\begin{align}
  \textrm{$\beta(\varnothing)=\varnothing$ and $\beta(S_p^{[p+1,p+h]})\subseteq\STRICT^{[p+1,p+h]}$}. 
\label{goodbeta}
\end{align}
Note $S_p^{[p+1,p+h]}=\STRICT^{[p+1,p+h]}$ (see Lemma \ref{ideal3}).
We have a well-defined injection 
\begin{align*}
  \BETAAA_{\GOOD}:S_p\LONGHOOKRIGHTARROW \STRICT,\quad
  \lambda \LONGMAPSTO \dots\sqcup\beta^{(2)}(\lambda|_{[2p-h,2p+h]})\sqcup\beta^{(1)}(\lambda|_{[p-h,p+h]})\sqcup\beta^{(0)}(\lambda|_{[1,h]}),
\end{align*}
by Lemma \ref{invshift} and Lemma \ref{weakideal}, where
\begin{align}
\GOOD^{(j)} = \SHIFT^{(j-1)}_{p}\circ\GOOD\circ\SHIFT^{-(j-1)}_{p}:S_p^{[jp-h,jp+h]}\LONGHOOKRIGHTARROW \STRICT^{[jp-h,jp+h]}
  \label{betajei}
\end{align}
for $j\geq 0$. 
Note that $\SHIFT_p^k(\varnothing)=\varnothing$ for $k\in\mathbb{Z}$.

\begin{Def}
We put $\SCHUR_{p,\GOOD}=\BETAAA_{\GOOD}(S_p)$ for $\GOOD$ that
satisfies \eqref{goodbeta}.
\end{Def}

\begin{Rem}\label{SpScpPT}
If $\GOOD$ is size preserving, then obviously we have $\SCHUR_{p,\GOOD}\PT S_p$.
\end{Rem}

\subsection{Pattern avoidance}
In this subsection, we see that
$\SCHUR_{p,\GOOD}$ is characterized by pattern avoidance
for $\GOOD$ satisfying certain conditions (see Proposition \ref{PROPSCHUR}).

\begin{Def}\label{gooddef}
  We say that a map $\GOOD:S_p^{[p-h,p+h]}\to\STRICT^{[p-h,p+h]}$ is good
  if 
\begin{align}
  \textrm{$\varnothing\in\FSS_{p,\GOOD}^{(i)}$, and $(\lambda_1,\dots,\lambda_\ell)\in\FSS_{p,\GOOD}^{(i)}$ implies $(\lambda_q,\dots,\lambda_r)\in\FSS_{p,\GOOD}^{(i)}$}
\label{weakcC}
\end{align}
for $i=1$ and $1\leq q\leq r\leq\ell$. Here, for $i\geq 1$, we put 
  \begin{align*}
    \FSS_{p,\GOOD}^{(i)}=\{\GOOD^{(i+1)}(\nu|_{[(i+1)p-h,(i+1)p+h]})\sqcup\GOOD^{(i)}(\nu|_{[ip-h,ip+h]})
    \mid \nu\in S_p^{[ip-h,(i+1)p+h]}\}. 
  \end{align*}
\end{Def}

\begin{Rem}\label{weakw}
  Note that \eqref{weakcC} holds for $i\geq 1$ if we have it for $i=1$, i.e., $\GOOD$ is good,
  because we have $\SHIFT^a_p(\FSS^{(b)}_{p,\GOOD})= \FSS^{(a+b)}_{p,\GOOD}$ 
for any map $\GOOD:S_p^{[p-h,p+h]}\to\STRICT^{[p-h,p+h]}$, where $a\in\mathbb{Z}$ and $b\geq 1$ such that
$a+b\geq 1$. This follows from the obvious formula
\begin{align*}
  \SHIFT_p^a(\GOOD^{(b+1)}(\mu)\sqcup\GOOD^{(b)}(\nu))
  =\GOOD^{(a+b+1)}(\SHIFT_p^{a}(\mu))\sqcup\GOOD^{(a+b)}(\SHIFT_p^{a}(\nu)))
\end{align*}
for $\mu\in S_p^{[(b+1)p-h,(b+1)p+h]}$, $\nu\in S_p^{[bp-h,bp+h]}$ and
Lemma \ref{invshift}.
Thus, we have $\SHIFT^a_p(\FS^{(b)}_{p,\GOOD})= \FS^{(a+b)}_{p,\GOOD}$. Here, for $i\geq 1$, we define a finite set $\FS^{(i)}_{p,\GOOD}$ by
\begin{align*}
  \FS^{(i)}_{p,\GOOD} = \STRICT^{[ip-h,(i+1)p+h]}\setminus\FSS^{(i)}_{p,\GOOD}.
\end{align*}
\end{Rem}

\begin{Rem}\label{weakw3}
  Assume that a map $\GOOD:S_p^{[p-h,p+h]}\to\STRICT^{[p-h,p+h]}$ satisfies $\beta(S_p^{[p+1,p+h]})\subseteq\STRICT^{[p+1,p+h]}$.
  Since $\GOOD^{(0)}$ is well-defined,
  $\FSS^{(0)}_{p,\GOOD}$ (resp. $\FS^{(0)}_{p,\GOOD}$) is well-defined by the same formula for $\FSS^{(i)}_{p,\GOOD}$ (resp. $\FS^{(i)}_{p,\GOOD}$) for $i\geq 1$.
  It is also easy to see
  $\SHIFT_p(\FSS^{(0)}_{p,\GOOD})\subseteq\FSS^{(1)}_{p,\GOOD}$.
\end{Rem}

\begin{Rem}\label{weakw34}
  In Remark \ref{weakw3}, 
  assume further that the restriction $\beta|_{S_p^{[p+1,p+h]}}:S_p^{[p+1,p+h]}\to\STRICT^{[p+1,p+h]}$
  is bijective. 
  Then, we easily have
  $\SHIFT_p(\FS^{(0)}_{p,\GOOD})\subseteq\FS^{(1)}_{p,\GOOD}$.
  For completeness, we give a proof in the rest of this Remark. It is enough to show that
  $\SHIFT_p(\lambda)=\FSS^{(1)}_{p,\GOOD}$ for $\lambda\in\STRICT^{[1,p+h]}$ implies
  $\lambda\in\FSS^{(0)}_{p,\GOOD}$. Let $\SHIFT_p(\lambda)=\GOOD^{(2)}(\mu)\sqcup\GOOD(\nu)$
  for $\mu\in S_p^{[2p-h,2p+h]}$ and $\nu\in S_p^{[p-h,p+h]}$ such that $\mu\sqcup\nu\in S_p$.
  By $\SHIFT_p(\lambda)\in\STRICT^{[p+1,2p+h]}$ and the assumption of $\GOOD$, we have $\nu\in S_p^{[p+1,p+h]}$ and thus we have $\lambda=\GOOD(\SHIFT_p^{-1}(\mu))\sqcup \GOOD^{(0)}(\SHIFT_p^{-1}(\nu))$.
  By Lemma \ref{invshift}, we have $\SHIFT_p^{-1}(\mu)\sqcup\SHIFT_p^{-1}(\nu)=\SHIFT_p(\mu\sqcup\nu)\in S_p$ and thus we have $\lambda\in\FSS^{(0)}_{p,\GOOD}$.
\end{Rem}

  
\begin{Cor}\label{weakcc}
Assume that a map $\GOOD:S_p^{[p-h,p+h]}\to\STRICT^{[p-h,p+h]}$ is good in the sense of Definition \ref{gooddef}.  
If a strict partition $\lambda$ contains $\SHIFT^{j}_p(\mu)$
for $j\geq 0$ and $\mu\in \FS^{(1)}_{p,\GOOD}$, then we have $\lambda|_{[(j+1)p-h,(j+2)p+h]}\in\FS^{(j+1)}_{p,\GOOD}$.
\end{Cor}

\begin{proof}
  Let $\nu=\lambda|_{[(j+1)p-h,(j+2)p+h]}$ and assume $\nu\in\FSS^{(j+1)}_{p,\GOOD}$. 
  By Remark \ref{weakw}, $\SHIFT_p^{-j}(\nu)\in\FSS^{(1)}_{p,\GOOD}$ and it contains $\mu\in\FS^{(1)}_{p,\GOOD}$. It contradicts the goodness of $\GOOD$.
\end{proof}

\begin{Lem}\label{weakc3}
  Assume that a map $\GOOD:S_p^{[p-h,p+h]}\to\STRICT^{[p-h,p+h]}$ is good in the sense of Definition \ref{gooddef}, satisfies $\GOOD(S_p^{[p+1,p+h]})\subseteq\STRICT^{[p+1,p+h]}$ and
  $\GOOD|_{S_p^{[p+1,p+h]}}:S_p^{[p+1,p+h]}\to\STRICT^{[p+1,p+h]}$ is bijective.
  We have \eqref{weakcC} for $i=0$ (see Remark \ref{weakw3}).
\end{Lem}

\begin{proof}
  Because $\varnothing\in\FSS^{(0)}_{p,\GOOD}$ is obvious, we prove the latter condition
  in \eqref{weakcC} for $i=0$.
  Let $\lambda=(\lambda_1,\dots,\lambda_{\ell})\in\FSS^{(0)}_{p,\GOOD}$ and assume
  $\lambda'=(\lambda_q,\dots,\lambda_r)\in\FS^{(0)}_{p,\GOOD}$. We have $\SHIFT_p(\lambda')\in\FS^{(1)}_{p,\GOOD}$
  and $\SHIFT_p(\lambda)\in\FSS^{(1)}_{p,\GOOD}$ by Remark \ref{weakw3} and Remark \ref{weakw34}, which contradicts
  the goodness of $\GOOD$. Thus, we have $\lambda'\in\FSS^{(0)}_{p,\GOOD}$.
\end{proof}

\begin{Prop}\label{PROPSCHUR}
  Assume that a map $\GOOD:S_p^{[p-h,p+h]}\to\STRICT^{[p-h,p+h]}$ is injective, satisfies \eqref{goodbeta} and good in the sense of Definition \ref{gooddef}.
For $\lambda\in\STRICT$, we have $\lambda\in\SCHUR_{p,\GOOD}$ if and only if $\lambda$ does not contain any $\SHIFT^{k}_p(\mu)$
for $k\geq 0$ and $\mu\in \FS^{(1)}_{p,\GOOD}\cup\FS^{(0)}_{p,\GOOD}$. 
\end{Prop}

\begin{proof}
  Note that the injectivity of $\GOOD$ and \eqref{goodbeta} imply
  that $\GOOD(S_p^{[p+1,p+h]})\subseteq\STRICT^{[p+1,p+h]}$ and
  $\GOOD|_{S_p^{[p+1,p+h]}}:S_p^{[p+1,p+h]}\to\STRICT^{[p+1,p+h]}$ is bijective.
  Note also that, for $\lambda\in\STRICT$, we have $\lambda\in\SCHUR_{p,\GOOD}$ if and only if $\lambda^{\LANGLEL j\RANGLER}=\lambda|_{[(j-1)p-h,jp+h]}$ is not in $\FS^{(j-1)}_{p,\GOOD}$
  (i.e., $\lambda^{\LANGLEL j\RANGLER}\in\FSS^{(j-1)}_{p,\GOOD}$) for all $j\geq 1$ by 
  Lemma \ref{invweakideal} and the injectivity of $\beta$.  

  Assume $\lambda\not\in \SCHUR_{p,\GOOD}$, and take $j$ such that $\lambda^{\LANGLEL j\RANGLER}\in\FS^{(j-1)}_{p,\GOOD}$.
  If $j\geq 2$ (resp $j=1)$, then $\lambda$ contains $\lambda^{\LANGLEL j\RANGLER}=\SHIFT_p^{j-2}(\SHIFT_p^{2-j}(\lambda^{\LANGLEL j\RANGLER}))$ (resp. $\lambda^{\LANGLEL 1\RANGLER}=\SHIFT^0_p(\lambda^{\LANGLEL 1\RANGLER})$), where note $\SHIFT_p^{2-j}(\lambda^{\LANGLEL j\RANGLER})\in\FS^{(1)}_{p,\GOOD}$ (resp.
  $\lambda^{\LANGLEL 1\RANGLER}\in\FS^{(0)}_{p,\GOOD}$) by Remark \ref{weakw}.
  
  Assume that $\lambda$ contains $\SHIFT^{k}_p(\mu)$
  for $k\geq 0$ and $\mu\in \FS^{(1)}_{p,\GOOD}\cup\FS^{(0)}_{p,\GOOD}$.
  If $\mu\in\FS^{(1)}_{p,\GOOD}$, then $\lambda^{\LANGLEL k+2\RANGLER}\in\FS^{(k+1)}_{p,\GOOD}$ by Corollary \ref{weakcc} and we have $\lambda\not\in\SCHUR_{p,\GOOD}$.
  If $\mu\in\FS^{(0)}_{p,\GOOD}$ and $k\geq 1$, then $\SHIFT^k_p(\mu)\in\SHIFT_p^{k-1}(\FS^{(1)}_{p,\GOOD})$ by Remark \ref{weakw34} and we have $\lambda\not\in\SCHUR_{p,\GOOD}$.
  If $\mu\in\FS^{(0)}_{p,\GOOD}$ and $k=0$, then $\lambda^{\LANGLEL 1\RANGLER}\in\FS^{(0)}_{p,\GOOD}$ by Lemma \ref{weakc3} and we have $\lambda\not\in\SCHUR_{p,\GOOD}$.
\end{proof}

\subsection{Definition of $\SCHUR_p$}\label{defsch2}
In this subsection, we
fix $j\geq 0$ as in \S\ref{magicmapsec}. Recall
$b_i=b_i(\mu_+,\mu_-)$ for $\mu_+\in S_p^{[jp+1,jp+h]}$ and $\mu_-\in S_p^{[jp-h,jp-1]}$
in Definition \ref{invdef}.

For $k=0,\dots,b$, where $b=b(\mu_+,\mu_-)=\max\{b_i\mid 0\leq i\leq h\}$,
we define
\begin{align*}
\gamma_k= \min \{0\leq i\leq h\mid b_i=k\},
\end{align*}
and promise $\gamma_{b+1}=h$.
This gives a decomposition of the half interval
\begin{align}
  (0,h]=(\gamma_0,\gamma_1]\sqcup 
      \dots
        \sqcup(\gamma_b,\gamma_{b+1}], \textrm{ where $(c,d]=\{x\in\mathbb{R}\mid c<x\leq d\}$}.
\label{decompinterval}
\end{align}
Note that $\gamma_0=0$, and the last half interval $(\gamma_b,\gamma_{b+1}]$ may be empty.
\begin{Def}\label{invdefstr}
For $\mu_+\in S_p^{[jp+1,jp+h]}$ and $\mu_-\in S_p^{[jp-h,jp-1]}$, the formula
\begin{align}
\begin{split}
  \INVMAPSTR{j}(\mu_+,\mu_-)_{jp\pm i}=\begin{cases}
0 & \textrm{if $i=0$ and $b\in 2\mathbb{Z}$}, \\
1 & \textrm{if $i=0$ and $b\not\in 2\mathbb{Z}$}, \\
1-m_{jp\mp i}(\mu_{\mp}) & \textrm{if $i\in(\gamma_{k},\gamma_{k+1}]$ and $k-b\not\in 2\mathbb{Z}$}, \\
m_{jp\pm i}(\mu_{\pm}) & \textrm{if $i\in(\gamma_{k},\gamma_{k+1}]$ and $k-b\in 2\mathbb{Z}$},
  \end{cases}
\end{split}
\label{opstr}
\end{align}
defines $\INVMAPSTR{j}(\mu_+,\mu_-)\in \STRICT^{[jp-h,jp+h]}$, where $0\leq i\leq h$
and $0\leq k\leq b$.
\end{Def}

\begin{Rem}\label{b01}
We have
$\INVMAPSTR{j}(\mu_+,\mu_-)=\INVMAP{j}(\mu_+,\mu_-)$ if $b=0,1$. 
\end{Rem}

There is an operational description 
for $\INVMAPSTR{j}(\mu_+,\mu_-)$ similar to $\INVMAP{j}(\mu_+,\mu_-)$.

\begin{Ex}
  Take $p$, $j$, $\mu_{\pm}$ as in Example \ref{pict}.
The picture below gives
$\gamma_0=0$, $\gamma_1=4$, $\gamma_2=8$, $\gamma_3=14$, $\gamma_4=17$,
and
\begin{align*}
  \INVMAPSTR{j}(\mu_+,\mu_-)=(44,42,41,40,34,33,32,31,27,26,25,24,23,21,18).
  \end{align*}
\begin{align*}
  \begin{array}{l|ccccccccccccccccc}
    i & 16 & 15 & 14 & 13 & 12 & 11 & 10 & 9 & 8 & 7 & 6 & 5 & 4 & 3 & 2 & 1 & 0 \\ \hline
    b_i & 2 & 3 & 3 & 2 & 1 & 1 & 1 & 1 & 2 & 1 & 0 & 0 & 1 & 0 & -1 & -1 & 0 \\ \hline
    \mu_+ & \circ & \circ & 47 & 46 & \circ & 44 & \circ & \circ & 41 & 40 & \circ & \circ & 37 & 36 & \circ & \circ \\
    \mu_- & \circ & 18 & 19 & 20 & 21 & \circ & 23 & \circ & 25 & 26 & 27 & \circ & 29 & 30 & 31 & \circ \\ \hline
    \phantom{\mu_+} & \circ & \circ & \circ & \circ & \circ & 44 & \circ & 42 & 41 & 40 & \circ & \circ & \circ & \circ & \circ & 34 \\
    \phantom{\mu_-} & \circ & 18 & \circ & \circ & 21 & \circ & 23 & 24 & 25 & 26 & 27 & \circ & \circ & \circ & 31 & 32  
\end{array}
\end{align*}
\end{Ex}

\begin{Prop}\label{oneboxstrict}
Let $j\geq 1$, and assume $\KF_i(\lambda\otimes\mu)=\lambda'\otimes\mu'$
for $\lambda,\lambda'\in S_p^{[jp+1,jp+h]}$, $\mu,\mu'\in S_p^{[jp-h,jp-1]}$ and
$i\in \AII{p}$ (recall \eqref{aidex}).
Then, $\INVMAPSTR{j}(\lambda',\mu')$ is obtained by adding a residue $i$-box to
$\INVMAPSTR{j}(\lambda,\mu)$.
\end{Prop}

\begin{proof}
  We consider 4 cases (\ref{caseX}), (\ref{caseY}), (\ref{caseZ}), (\ref{caseW})
  in the proof of Proposition \ref{onebox}.
  For the case (\ref{caseX}), we need to consider 3 cases: $b=0$, $b\geq 1$ is odd, and $b\geq 1$ is even.
  Because the argument is straightforward and similar to the proof of Proposition \ref{onebox}, we omit a detail.
  The other cases are similar.
\end{proof}



\begin{Lem}\label{nobiru}
For $\mu_+\in S_p^{[jp+1,jp+h]}$ and $\mu_-\in S_p^{[jp-h,jp-1]}$, we have
\begin{enumerate}
\item\label{NAGASAAUX} $\ell(\INVMAP{j}(\mu_+,\mu_-))=\ell(\INVMAPSTR{j}(\mu_+,\mu_-))=\ell(\mu_+)+\ell(\mu_-)-b(\mu_+,\mu_-)$ (we denote it by $\ell)$,
\item\label{SIZEAUX} $|\INVMAP{j}(\mu_+,\mu_-)|=|\INVMAPSTR{j}(\mu_+,\mu_-)|=|\mu_+|+|\mu_-|-b(\mu_+,\mu_-)jp$,
\item\label{PSIAUX} $\INVMAPSTR{j}(\mu_+,\mu_-)_1\geq \INVMAP{j}(\mu_+,\mu_-)_1$ and $\INVMAPSTR{j}(\mu_+,\mu_-)_{\ell}\leq \INVMAP{j}(\mu_+,\mu_-)_{\ell}$ if $\ell\geq 1$.
\end{enumerate}
\end{Lem}

\begin{proof}
Let $b=2b'+1$. Then, the operation for $i\in(\gamma_{k},\gamma_{k+1}]$ in \eqref{opstr}, where $k=0,2,\dots,2b'$, decreases the length by 2 (resp. size by $2jp$).
  Because the part $jp$ is added, the length (resp. size) is decreased by $2(b'+1)-1=b$
  (resp. $bjp$) as
  a total. The case $b\in 2\mathbb{Z}$ is similar.
  This proves (\ref{NAGASAAUX}), (\ref{SIZEAUX}) for $\INVMAPSTR{j}$, and
  proof for $\INVMAP{j}$ is similar.
  To prove (\ref{PSIAUX}), 
  we may assume $\mu_+\ne\varnothing$ since $\INVMAPSTR{j}(\varnothing,\mu_-)=\INVMAP{j}(\varnothing,\mu_-)=\mu_-$. Let $(\mu_+)_1=jp+x$ for $1\leq x\leq h$.
  We may assume $b\geq 2$ by Remark \ref{b01}, and may assume $x=\gamma_b$.
  If there exists $i\in(\gamma_{b-1},\gamma_b]$ such that
    $jp-i\not\in\mu_{-}$, then we have $\INVMAPSTR{j}(\mu_+,\mu_-)_1=\INVMAP{j}(\mu_+,\mu_-)_1$. Otherwise, we have $\INVMAPSTR{j}(\mu_+,\mu_-)_1=jp+\gamma_{b-1}> \INVMAP{j}(\mu_+,\mu_-)_1$. 
    The inequality $\INVMAPSTR{j}(\mu_+,\mu_-)_{\ell}\leq \INVMAP{j}(\mu_+,\mu_-)_{\ell}$ is proved similarly.
\end{proof}

\begin{Prop}\label{injstrict}
The map $\INVMAPSTR{j}:S_p^{[jp+1,jp+h]}\times S_p^{[jp-h,jp-1]}\to \STRICT^{[jp-h,jp+h]}$
is injective.
\end{Prop}

\begin{proof}
For $\mu_+\in S_p^{[jp+1,jp+h]}$ and $\mu_-\in S_p^{[jp-h,jp-1]}$,
take the decomposition \eqref{decompinterval}, and put
$\lambda=\INVMAPSTR{j}(\mu_+,\mu_-)$.
Recall $a_i(\lambda)$ in Definition \ref{magicmap}.
It is easy to see 
\begin{align*}
a_i(\lambda) = \begin{cases}
1 & \textrm{if $i=0$ and $b$ is odd}, \\
0 & \textrm{if $i=0$ and $b$ is even}, \\
b_i(\mu_+,\mu_-)-(k-1) & \textrm{if $i\in(\gamma_{k-1},\gamma_{k}]$ for $1\leq k\leq b+1$ such that $b-k\not\in 2\mathbb{Z}$}, \\
k-b_i(\mu_+,\mu_-) & \textrm{if $i\in(\gamma_{k-1},\gamma_{k}]$ for $1\leq k\leq b+1$ such that $b-k\in 2\mathbb{Z}$},
  \end{cases}
\end{align*}
where $0\leq i\leq h$. 
In particular, we have $\gamma_k=\gamma'_k$ if we define
\begin{align*}
  \gamma'_k = \min(\{\gamma'_{k-1}<i\leq h\mid a_i(\lambda)
  = \PARITY(a_0(\lambda)-k)\}\cup\{h\})
\end{align*}
for $k\geq 1$ with $\gamma'_{k-1}<h$, where $\gamma'_0=0$, and $\PARITY$ of an even (resp. odd)
integer is defined to be 0 (resp. 1).
These imply that $\mu_\pm$ can be recovered from $\lambda$,
similarly to that $\mu_{\pm}$ can be recovered from $\INVMAP{j}(\mu_+,\mu_-)$
as in \S\ref{magicmapsec}.
\end{proof}

\begin{Cor}\label{binj}
  There exists a unique size and length preserving injection $\GOOD_{p,j}:S_p^{[jp-h,jp+h]}\to \STRICT^{[jp-h,jp+h]}$
such that $\INVMAPSTR{j}(\mu_+,\mu_-)=\GOOD_{p,j}(\INVMAP{j}(\mu_+,\mu_-))$
for $\mu_+\in S_p^{[jp+1,jp+h]}$ and $\mu_-\in S_p^{[jp-h,jp-1]}$.
\end{Cor}

\begin{proof}
This is immediate by Theorem \ref{pmbij}, Lemma \ref{nobiru} and Proposition \ref{injstrict}.
\end{proof}

\begin{Rem}\label{s1h}
Note that, by constructions of $\INVMAP{j}$ and $\INVMAPSTR{j}$, we have
$\GOOD_{p,j}= \GOOD_{p,1}^{(j)}$ (see \eqref{betajei}), and 
$\GOOD_{p,1}(\lambda)=\lambda$ for $\lambda\in S_p^{[p+1,p+h]}$.
\end{Rem}

\begin{Def}
  For an odd integer $p\geq 3$, we put $\SCHUR_p=\SCHUR_{p,\GOOD_{p}}$,
  where $\GOOD_p=\GOOD_{p,1}$.
\end{Def}

\begin{Rem}\label{remscpPT}
  Through the bijection $\BETAAA_{\GOOD_{p}}:S_p\ISOM \SCHUR_p$, we regard $\SCHUR_p$ as an
  $A^{(2)}_{p-1}$-crystal, which is isomorphic to $B(\Lambda_0)$.
By Remark \ref{SpScpPT}, we have $\SCHUR_p\PT S_p$.
\end{Rem}

\subsection{Examples}\label{exschur}
By Remark \ref{b01}, we have $\GOODZERO{p}(\lambda)=\lambda$ if $m_{p}(\lambda)\leq 1$.
As shown below, $\SCHUR_{p}$ for $p=3,5$ is nothing but $S_3$ and $\SCHUR_5$ in Theorem \ref{BMOGen}, respectively.

\begin{enumerate}
\item For $p=3$, we have $\GOODZERO{3}=\ID_{S_3^{[2,4]}}$.
\item For $p=5$, we have $\GOODZERO{5}((5,5))=(6,4)$.
Note that $\BETAAA_{\GOOD_{5}}=\BETAA_5$ in \S\ref{expschur}. 
\end{enumerate}

  For $p=7$, we have $\GOODZERO{7}((7,7))=(8,6)$, $\GOODZERO{7}((7,7,7))=(9,7,5)$,
$\GOODZERO{7}((7,7)\sqcup(a))=(9,5)\sqcup(a)$,
$\GOODZERO{7}((7,7)\sqcup(b))=(8,6)\sqcup(b)$, 
where $a=8,6$ and $b=10,9,5,4$. 

\subsection{Properties of $\SCHUR_p$}

\begin{Lem}\label{ichidan}
  For an odd integer $p=2h+1\geq 3$, we have $(\lambda_q,\dots,\lambda_{r})\in\GOOD_{p}(S_p^{[p-h,p+h]})$
  for $(\lambda_1,\dots,\lambda_{\ell})\in\GOOD_{p}(S_p^{[p-h,p+h]})$ and $1\leq q\leq r\leq\ell$.
\end{Lem}

\begin{proof}
  It is enough to show $(\lambda_2,\dots,\lambda_{\ell})\in\GOOD_{p}(S_p^{[p-h,p+h]})$ and
  $(\lambda_1,\dots,\lambda_{\ell-1})\in\GOOD_{p}(S_p^{[p-h,p+h]})$ for 
  $\lambda=(\lambda_1,\dots,\lambda_\ell)=\INVMAPSTR{1}(\mu_+,\mu_-)$,
  where $\mu_+\in \STRICT^{[p+1,p+h]}$ and
  $\mu_-\in \STRICT^{[p-h,p-1]}$,
  if $\ell\geq 1$, i.e., $(\mu_+,\mu_-)\ne(\varnothing,\varnothing)$.
  Because both proof are similar, we show the former.
  We write $\nu'=\nu\setminus\{\nu_1\}$ for a nonempty strict partition $\nu$.
  
  Because of
  $\INVMAPSTR{1}(\varnothing,\mu_-)=\mu_-$, we may assume
  $\mu_+\ne\varnothing$. 
  Take the decomposition \eqref{decompinterval}.
  We may assume $b\geq 1$
  because otherwise we have $\lambda=\mu_+\sqcup\mu_-=\INVMAPSTR{1}(\mu_+,\mu_-)$
  and $\lambda'=(\mu_+)'\sqcup\mu_-=\INVMAPSTR{1}((\mu_+)',\mu_-)$ by Remark \ref{b01}.
  We may also assume
  $\lambda_1<p+\gamma_b$ because otherwise we have
  $\lambda_1>p+\gamma_b$ and $\lambda'=\INVMAPSTR{1}((\mu_+)',\mu_-)$.

  If $\lambda_1=p+\gamma_{b-1}$, then it is easy to see $\lambda'=\INVMAPSTR{1}(\mu^{\circ}_+,\mu^{\circ}_-)$, where (see \S\ref{noco})
  \begin{align*}
    \mu^{\circ}_-=\mu_-\setminus\{p-\gamma_b\},\quad
    \mu^{\circ}_+=
    \begin{cases}
      (\mu_+\setminus\{p+\gamma_b\})\setminus\{p+\gamma_{b-1}\} & \textrm{if $b\geq 2$,}\\
      \mu_+\setminus\{p+\gamma_b\} & \textrm{if $b=1$.}
    \end{cases}
  \end{align*}
  The remaining case is $\lambda_1=p+x$ for some $\gamma_{b-1}<x<\gamma_{b}$.
  Put 
  \begin{align*}
    \mu^{\circ}_-=(\mu_-\setminus\{p-\gamma_b\})\sqcup\{p-x\},\quad
    \mu^{\circ}_+=\mu_+\setminus\{p+\gamma_b\}.
  \end{align*}
    It is not difficult to see $\lambda'=\INVMAPSTR{1}(\mu^{\circ}_+,\mu^{\circ}_-)$.
\end{proof}

\begin{Prop}\label{propgoodbeta}
  For an odd integer $p\geq 3$, $\GOOD_{p}$
  is good (see Definition \ref{gooddef}).
\end{Prop}

\begin{proof}
  It is enough to show $(\lambda_2,\dots,\lambda_{\ell})\in\FSS^{(1)}_{p,\GOOD_{p}}$ 
  and $(\lambda_1,\dots,\lambda_{\ell-1})\in\FSS^{(1)}_{p,\GOOD_{p}}$ for 
  \begin{align*}
    \lambda=(\lambda_1,\dots,\lambda_\ell)=\INVMAPSTR{2}(\mu_+,\mu_-)\sqcup\INVMAPSTR{1}(\nu_+,\nu_-),
  \end{align*}
  where $\mu_+\in \STRICT^{[2p+1,2p+h]}$,
  $\mu_-\in \STRICT^{[2p-h,2p-1]}$,
  $\nu_+\in \STRICT^{[p+1,p+h]}$,
  $\nu_-\in \STRICT^{[p-h,p-1]}$,
  if $\ell\geq 1$. Because both proof are similar, we show the former.

  By Lemma \ref{ichidan}, we may assume $2p-h\leq\lambda_1\leq2p+h$.
  Therefore, the goal is showing, by virtue of Corollary \ref{zigzagtotal},
  $(\mu^{\circ}_-,\nu_+)\in S_p$ assuming
  $(\mu_-,\nu_+)\in S_p$
  for
\begin{align*}
  (\lambda_2,\dots,\lambda_{\ell})=\INVMAPSTR{2}(\mu^{\circ}_+,\mu^{\circ}_-)\sqcup\INVMAPSTR{1}(\nu_+,\nu_-).
\end{align*}

  Take the decomposition \eqref{decompinterval} for $(\mu_+,\mu_-)$.
  By the proof of Lemma \ref{ichidan}, we have 
  \begin{align*}
    \textrm{$\mu^{\circ}_-=\mu_-$\quad
      or\quad
      $\mu^{\circ}_-=\mu_-\setminus\{2p-\gamma_b\}$\quad
      or\quad
    $\mu^{\circ}_-=(\mu_-\setminus\{2p-\gamma_b\})\sqcup\{2p-x\}$,}
  \end{align*}
  where the second and the third equalities hold only if $b\geq 1$, and
  the third equality holds only if $\lambda_1=2p+x$ and $\gamma_{b-1}<x<\gamma_b$.
  In each case, it is easy to see $(\mu^{\circ}_-,\nu_+)\in S_p$ by \eqref{charS} and Proposition \ref{indnum}.
\end{proof}

\begin{Prop}\label{proofPSI}
Any $\lambda\in\SCHUR_p$ satisfies the $p$-Schur inequality. 
\end{Prop}

\begin{proof}
  For $1\leq i\leq\ell(\lambda)-h$, we show $\lambda_i-\lambda_{i+h}\geq p$, and
  the inequality is strict if $\lambda_i\in p\mathbb{Z}$. Let $\lambda'=(\lambda_i,\dots,\lambda_{i+h})$.
  We may assume (A) or (B) below.
  \begin{enumerate}
  \item[(A)] $jp-h\leq\lambda_{i+h}, \lambda_i\leq jp+h$ for some $j\geq 0$.
  \item[(B)] $jp-h\leq\lambda_{i+h}\leq jp+h$ and $(j+1)p-h\leq\lambda_i\leq (j+1)p+h$ for some $j\geq 0$.
  \end{enumerate}
  Note $\lambda'\in\SCHUR_p$ by Proposition \ref{PROPSCHUR}, Corollary \ref{binj}, Remark \ref{s1h} and
  Proposition \ref{propgoodbeta}.
  Assume (A), and take $\mu\in S_p^{[jp-h,jp+h]}$ such that $\GOOD_p^{(j)}(\mu)=\lambda'$.
  By Lemma \ref{nobiru}, we have $\ell(\mu)=h+1$ and
  $\lambda_i-\lambda_{i+h}\geq \mu_1-\mu_{h+1}$, where
  we have the equality if and only if $\lambda_i=\mu_1$ and $\lambda_{i+h}=\mu_{h+1}$.
  Therefore, our goal follows from the $p$-Schur inequality for $\mu$. 
The case (B) is similar.
\end{proof}

\subsection{Proof of Theorem \ref{maintheorem3}}
(\ref{maintheorem3a}) (resp. (\ref{maintheorem3b})) is
the same as Proposition \ref{proofPSI} (resp. Proposition \ref{PROPSCHUR}, Corollary \ref{binj}, Remark \ref{s1h}, Proposition \ref{propgoodbeta}).
For $\lambda\in\SCHUR_p$ and $i\in\AI{p}\setminus\{h\}$, 
$\KF_i\lambda$ is obtained by adding a residue $i$-box to $\lambda$ 
by Proposition \ref{oneboxstrict}, if defined.
It is also valid for $i=h$
by the same argument as that for Theorem \ref{finonebox} (i.e.,
the $\INVMAPSTR{j}$-version of \eqref{onebox3} is clearly
valid 
under the conditions \eqref{onebox1}, \eqref{onebox2}).






\begin{thebibliography}{99}
\bibitem{ABO}
G.E.~Andrews, C.~Bessenrodt and J.B.~Olsson,
\textit{Partition identities and labels for some modular characters}, 
Trans.Amer.Math.Soc. \textbf{344} (1994), 597--615.

\bibitem{An1}
G.E.~Andrews, \textit{The theory of partitions}, 
Encyclopedia of Mathematics and its Applications, Vol. 2. Addison-Wesley Publishing Co., Reading, Mass.-London-Amsterdam, 1976.

\bibitem{An2}
G.E.~Andrews, \textit{$q$-series: their development and application in analysis, number theory, combinatorics, physics, and computer algebra},
CBMS Regional Conference Series in Mathematics, 66. Published for the Conference Board of the Mathematical Sciences, Washington, DC; by the American Mathematical Society, Providence, RI, 1986.

\bibitem{An3}
G.E.~Andrews, \textit{A general theory of identities of the Rogers-Ramanujan type}, 
Bull.Amer.Math.Soc. \textbf{80} (1974), 1033--1052.

\bibitem{An4}
G.E.~Andrews, \textit{On the general Rogers-Ramanujan theorem}, 
Memiors of the American Mathematical Society, \textbf{152}. American Mathematical Society, Providence, R.I., 1974.


\bibitem{Be1}
C.~Bessenrodt, \textit{A combinatorial proof of a refinement of the Andrews-Olsson partition identity}, European J.Combin. \textbf{12} (1991), 271--276.

\bibitem{Be2}
C.~Bessenrodt, \textit{Representations of the covering groups of the symmetric groups and their combinatorics}, Sem.Lothar.Combin. 33 (1994) (electronic).

\bibitem{BK}
J.~Brundan and A.~Kleshchev, 
\textit{Hecke-Clifford superalgebras, crystals of type $A^{(2)}_{2\ell}$ and modular branching rules for $\widehat{S}_n$},
Represent.Theory \textbf{5} (2001), 317--403. 

\bibitem{BK2}
J.~Brundan and A.~Kleshchev, \textit{Projective representations of symmetric groups via Sergeev duality},
Math.Z. \textbf{239} (2002) 27--68.

\bibitem{BK3}
J.~Brundan and A.~Kleshchev, \textit{James' regularization theorem for double covers of symmetric groups}, 
J.Algebra \textbf{306} (2006) 128--137. 

\bibitem{BMO}
C.~Bessenrodt, A.O.~Morris and J.B.~Olsson,
\textit{Decomposition matrices for spin characters of symmetric groups at characteristic 3}, 
J.Algebra \textbf{164} (1994), 146--172.

\bibitem{Bre}
D.M.~Bressoud, \textit{A combinatorial proof of Schur's 1926 partition theorem},
Proc.Amer.Math.Soc. \textbf{79} (1980), 338--340. 

\bibitem{Cra}
D.A.~Craven, 
\textit{Representation Theory of Finite Groups: A Guidebook},
Universitext. Springer, 2019.

\bibitem{Fig}
L.~Figueiredo,
\textit{Calculus of Principally Twisted Vertex Operators},
Mem.Amer.Math.Soc. \textbf{371} AMS (1987).

\bibitem{FOS}
G.~Fourier, M.~Okado and A.~Schilling,
\textit{Kirillov-Reshetikhin crystals for nonexceptional types}, 
Adv.Math. \textbf{222} (2009), 1080--1116.


\bibitem{Gr1}
I.~Grojnowski,
\textit{Affine $\hat{sl}_p$ controls the modular
representation theory of the symmetric group and
related Hecke algebras}, \textbf{math.RT/9907129}.

\bibitem{GV}
I.~Grojnowski and M.~Vazirani,
\textit{Strong multiplicity one theorems for affine Hecke algebras of type A},
Transform. Groups \textbf{6} (2001), 143--155.

\bibitem{Har}
G.H.~Hardy, \textit{Ramanujan}, Cambridge University Press, 1940.

\bibitem{FMP}
H.K.~Farahat, W.~M\"uller and M.H.~Peel,
\textit{The modular characters of the symmetric groups},
J.Algebra \textbf{40} (1976), 354--363.


\bibitem{HKOTY}
G.~Hatayama, A.~Kuniba, M.~Okado, T.~Takagi and Y.~Yamada, 
\textit{Remarks on fermionic formula}, 
Contemp. Math., \textbf{248} (1999), 243--291. 

\bibitem{HKOTZ}
G.~Hatayama, A.~Kuniba, M.~Okado, T.~Takagi and Z.~Tsuboi, 
\textit{Paths, crystals and fermionic formulae},
in ``MathPhys Odyssey 2001-Integrable Models and Beyond In Honor of Barry M.McCoy'', 
edited by M. Kashiwara and T. Miwa, Birkh"auser (2002) 205--272.

\bibitem{Isa}
M.~Isaacs, 
\textit{Character Theory of Finite Groups},
Academic Press, New York-London, 1976.

\bibitem{Ito}
K.~Ito,
\textit{Level 2 standard modules for $A^{(2)}_9$ and partition conditions of Kanade-Russell}, arXiv:2211.03652.

\bibitem{Jam}
G.D.~James, 
\textit{The irreducible representations of the symmetric groups}, 
Bull.London Math.Soc. \textbf{8} (1976), 229--232.

\bibitem{JMO}
N.~Jing, K.C.~Misra and M.~Okado,
\textit{$q$-wedge modules for quantized enveloping algebras of classical type},
J.Algebra \textbf{230} (2000), 518--539. 

\bibitem{Kac}
V.~Kac.
\textit{Infinite Dimensional Lie Algebras}.
Cambridge University Press, 1990.

\bibitem{Kas}
M.~Kashiwara,
\textit{Bases Cristallines des Groupes Quantiques},
Cours Sp\'ec., vol. 9, Soc. Math. France, 2002.


\bibitem{Kan}
S-J.~Kang,
\textit{Crystal bases for quantum affine algebras and combinatorics of Young walls},
Proc. London Math. Soc. \textbf{86} (2003), 29--69. 


\bibitem{KKT}
S-J.~Kang, M.~Kashiwara and S.~Tsuchioka,
\textit{Quiver Hecke superalgebras},
J.Reine.Angew.Math. \textbf{711} (2016), 1--54.


\bibitem{Kl2}
A.~Kleshchev, 
\textit{Linear and Projective Representations of Symmetric Groups}, Cambridge University Press, 2005.

\bibitem{KKMMNN2}
S-J.~Kang, M.~Kashiwara, K.~Misra, T.~Miwa, T.~Nakashima and A.~Nakayashiki, 
\textit{Affine crystals and vertex models}, Internat. J. Modern Phys. A \textbf{7}, Suppl. 1A (1992), 449--484.

\bibitem{KKMMNN}
S-J.~Kang, M.~Kashiwara, K.~Misra, T.~Miwa, T.~Nakashima and A.~Nakayashiki, 
\textit{Perfect crystals of quantum affine Lie algebras},
Duke Math. J. \textbf{68} (1992), 499--607. 


\bibitem{KN}
M.~Kashiwara and T.~Nakashima, \textit{Crystal graphs for representations of the $q$-analogue of classical Lie algebras},
J.Algebra \textbf{165} (1994), 295--345.

\bibitem{KNO}
M.~Kashiwara, T.~Nakashima and M.~Okado, 
\textit{Affine geometric crystals and limit of perfect crystals},
Trans. Amer. Math. Soc. \textbf{360} (2008), 3645--3686. 


\bibitem{KS}
M.~Kashiwara and Y.~Saito,
\textit{Geometric construction of crystal bases}, 
Duke Math. J. \textbf{89} (1997), 9--36.

\bibitem{Lei}
T.~Leinster, \textit{Basic Category Theory}, 
Cambridge Studies in Advanced Mathematics, 143. Cambridge University Press, Cambridge, 2014.


\bibitem{LT}
B.~Leclerc and J-Y.~Thibon, 
\textit{q-deformed Fock spaces and modular representations of spin symmetric groups},
J.Phys.A \textbf{30} (1997), 6163--6176. 

\bibitem{Lus}
G.~Lusztig, 
\textit{Introduction to Quantum Groups},
Reprint of the 1994 edition. Modern Birkha\"user Classics. Birkha\"user/Springer, New York, 2010.

\bibitem{LW}
J.~Lepowsky and R.L.~Wilson, 
\textit{The structure of standard modules. I. Universal algebras and the Rogers-Ramanujan identities},
Invent.Math. \textbf{77} (1984), 199--290. 

\bibitem{MM}
K.~Misra and T.~Miwa,
\textit{Crystal base for the basic representation of $U_q(\mathfrak{sl}(n))$}, 
Comm. Math. Phys. \textbf{134} (1990), 79--88.



\bibitem{OS}
M.~Okado and A.~Schilling,
\textit{Existence of Kirillov-Reshetikhin crystals for nonexceptional types},
Represent.Theory \textbf{12} (2008), 186--207.

\bibitem{OSS}
M.~Okado, A.~Schilling and M.~Shimozono,
\textit{A tensor product theorem related to perfect crystals},
J.Algebra \textbf{267} (2003), 212--245.

\bibitem{OV}
A.~Okounkov and A.~Vershik, 
\textit{A new approach to representation theory of symmetric groups}, 
Selecta Math. (N.S.) \textbf{2} (1996), 581--605.

\bibitem{Rog}
L.J.~Rogers, \textit{Second Memoir on the Expansion of certain Infinite Products},
Proc.London Math.Soc. S1-25 (1894), 318--343. 

\bibitem{Sch}
I.~Schur, \textit{Zur additiven Zahlentheorie}, S.-B. Preuss. Akad. Wiss. Phys.-Math.Kl, 1926, 488--495 (Reprinted in I.~Schur,Gesammelte Abhandlungen, vol.3, Springer, Berlin, 1973, 43--50).

\bibitem{Sch2}
I.~Schur, 
\textit{\"Uber die Darstellung der symmetrischen and der alternierenden Gruppe durch gebrochene lineare Substitutionen},
J. Reine Angew. Math. \textbf{139} (1911), 155--250. 

\bibitem{Sil}
A.V.~Sills, 
\textit{An Invitation to the Rogers-Ramanujan Identities. With a Foreword by George E. Andrews}, 
CRC Press, (2018).

\bibitem{Ste}
J.~Stembridge, 
\textit{Shifted tableaux and the projective representations of symmetric groups},
Adv. Math. \textbf{74} (1989), 87--134. 

\bibitem{Tsu}
S.~Tsuchioka, 
\textit{Hecke-Clifford superalgebras and crystals of type $D^{(2)}_{\ell}$},
Publ.Res.Inst.Math.Sci. \textbf{46} (2010), 423--471. 

\bibitem{Tsu2}
S.~Tsuchioka, 
\textit{A vertex operator reformulation of the Kanade-Russell conjecture modulo 9},
arXiv:2211.12351


\end{thebibliography}
\end{document}